\documentclass{amsart}

\usepackage{graphicx}

\usepackage{amsmath, amsthm, amssymb, stmaryrd}
\usepackage{xypic}
\usepackage[vcentermath,stdtext,enableskew]{youngtab}

\newcommand{\numpartsgtrn}{p_{\geq n}}
\newcommand{\parts}{\mathcal{P}}
\newcommand{\asscovercoroot}{\alpha_{vw}^\vee}

\newcommand{\coroot}[1]{\alpha_{#1}^\vee}
\newcommand{\Supp}{\mathrm{Supp}}
\newcommand{\supp}{\mathrm{supp}}
\newcommand{\Z}{\mathcal{Z}}
\newcommand{\Iaf}{I_{\mathrm{af}}}

\newcommand{\Waf}{W_\textrm{af}}
\newcommand{\ASSFB}[1]{\tilde{F}^{B_n}_{#1}}

\newcommand{\ASSFD}[1]{\tilde{F}^{D_n}_{#1}}
\newcommand{\specialgen}[1]{\mathbb{P}_{#1}}

\newcommand{\stat}{cc}
\newcommand{\genericstat}{\textrm{stat}}
\newcommand{\ASSF}[1]{\tilde{F}_{#1}}

\newcommand{\specialgenB}[1]{\mathbb{P}_{#1}^B}
\newcommand{\specialgenD}[1]{\mathbb{P}_{#1}^D}

\newcommand{\repkerB}{\Omega_{-1}^{B_n}}
\newcommand{\repkerBalg}{\Omega_{-1}^B}
\newcommand{\repkerD}{\Omega_{-1}^{D_n}}
\newcommand{\repkerDalg}{\Omega_{-1}^D}

\newcommand{\Gr}{\textrm{Gr}_G}
\newcommand{\GrB}{\textrm{Gr}_B}
\newcommand{\GrD}{\textrm{Gr}_D}

\newcommand{\coweight}[1]{\Lambda^\vee_{#1}}
\newcommand{\fincoweight}[1]{\omega^\vee_{#1}}

\newcommand{\funalc}{A_0}

\newcommand{\nilCox}{\mathbb{A}_0}
\newcommand{\ZD}{\mathcal{Z}^D}
\newcommand{\ZB}{\mathcal{Z}^B}

\newcommand{\WD}{\widetilde{D}_n}
\newcommand{\WB}{\widetilde{B}_n}
\newcommand{\Dstat}[1]{cc(#1)}

\newcommand{\Bgrass}{\widetilde{B}_n^0}

\newcommand{\Dgrass}{\widetilde{D}_n^0}
\newcommand{\Wgrass}{\widetilde{W}^0}

\newcommand{\B}{\mathbb{B}}

\newcommand{\PhiD}{\mathbf{\Phi}_D}

\newcommand{\PhiB}{\mathbf{\Phi}_B}
\newcommand{\PsiB}{\mathbf{\Psi}_B}

\newcommand{\Dpar}{\mathcal{P}_D^n}
\newcommand{\Dparonecolor}{(\mathcal{P}_D^{n})^{(1)}}

\newcommand{\Bpar}{\mathcal{P}_B^n}
\newcommand{\Bodd}[1]{\mathcal{P}_{odd}^{#1}}

\newcommand{\Aaf}{A_{\textrm{af}}}
\newcommand{\Piaf}{\Pi_{\textrm{af}}}
\newcommand{\Paf}{\Phat}
\newcommand{\Qaf}{\Qhat}

\newcommand{\g}{\mathfrak{g}}
\newcommand{\h}{\mathfrak{h}}

\newcommand{\cohomB}{\Gamma^{(n)}_B}
\newcommand{\homB}{\Gamma_{(n)}^B}
\newcommand{\cohomD}{\Gamma^{(n)}_D}
\newcommand{\homD}{\Gamma_{(n)}^D}

\newcommand{\Spin}{\textrm{Spin}}

\newcommand{\kSB}[1]{\tilde{G}^{B_n}_{#1}}

\renewcommand{\g}{\mathfrak{g}}
\newcommand{\gaf}{\mathfrak{g}_{\textrm{af}}}
\renewcommand{\h}{\mathfrak{h}}
\newcommand{\haf}{\mathfrak{h}_{\textrm{af}}}
\newcommand{\Qhat}{\hat{Q}}
\newcommand{\Phat}{\hat{P}}

\newcommand{\Wafext}{W_{\textrm{ext}}}

\newcommand{\barthree}{\bar{3}}
\newcommand{\bartwo}{\bar{2}}

\newtheorem{theorem}{Theorem}
\newtheorem{proposition}[theorem]{Proposition}
\newtheorem{lemma}[theorem]{Lemma}
\newtheorem{corollary}[theorem]{Corollary}
\newtheorem{conj}[theorem]{Conjecture}

\theoremstyle{definition}

\newtheorem{definition}[theorem]{Definition}
\newtheorem{remark}{Remark}
\newtheorem{example}[remark]{Example}

\numberwithin{theorem}{section}
\numberwithin{remark}{section}
\numberwithin{equation}{section}

\begin{document}
\title{Affine Stanley symmetric functions for classical types}

\author[S.~Pon]{Steven Pon}
\email{steven.pon@uconn.edu}
\urladdr{http://www.math.uconn.edu/~pon}
\address{Department of Mathematics \\ University of Connecticut\\ 196 Auditorium
Road Unit 3009, Storrs, CT 06269-3009, U.S.A.}

\begin{abstract} We introduce affine Stanley symmetric functions for the special orthogonal groups, a class of symmetric functions that model the cohomology of the affine Grassmannian, continuing the work of Lam and Lam, Schilling, and Shimozono on the special linear and symplectic groups, respectively.  For the odd orthogonal groups, a Hopf-algebra isomorphism is given, identifying (co)homology Schubert classes with symmetric functions.  For the even orthogonal groups, we conjecture an approximate model of (co)homology via symmetric functions.  In the process, we develop type $B$ and type $D$ non-commutative $k$-Schur functions as elements of the nilCoxeter algebra that model homology of the affine Grassmannian.  Additionally, Pieri rules for multiplication by special Schubert classes in homology are given in both cases.  Finally, we present a type-free interpretation of \emph{Pieri factors}, used in the definition of noncommutative $k$-Schur functions or affine Stanley symmetric functions for any classical type.

\end{abstract}

\maketitle

\section{Introduction}

\subsection{Stanley symmetric functions and Schubert polynomials}

In 1984, Stanley introduced \cite{stan84} what came to be known as the Stanley symmetric functions as a tool for studying the number of reduced words of the longest element of the symmetric group.  Stanley's symmetric functions were soon found to have a deep relation to the geometry of the flag manifold as the ``stable limit'' of the Schubert polynomials of Lascoux and Schutzenberger \cite{ls82-1, mac91, bjs93}.

A particularly fruitful point of view for analysis of Stanley symmetric functions was found in the nilCoxeter algebra by Fomin and Stanley \cite{fs94}.  Billey and Haiman \cite{bh94} later explored analogues of Schubert polynomials for all the classical types; that is, polynomial representatives for Schubert classes in the cohomology ring of $G/B$, where $G = SO(n, \mathbb{C})$ or $Sp(2n,\mathbb{C})$ and $B$ is a Borel subgroup.  They also studied analogues of Stanley symmetric functions that are stabilizations of type $B$ (resp. type $D$) Schubert polynomials.  (Independently, Fomin and Kirillov \cite{fk96} explored several different type $B$ analogues of Schubert polynomials by generalizing different geometric and combinatorial properties of the type $A$ polynomials and also derived type $B$ Stanley symmetric functions, defined in terms of the nilCoxeter algebra of the hyperoctahedral group, whose definition matches that of Billey and Haiman.)  T.K. Lam \cite{tklam96, tklam95} developed much of the combinatorics of types $B$ and $D$ Stanley symmetric functions using Kra\'skiewicz insertion, including proofs that both expand as non-negative integer combinations of Schur $P$-functions.

\subsection{The affine case}
More recently, Thomas Lam \cite{lam05} defined (type $A$) \emph{affine} Stanley symmetric functions, which he labeled as such because 1) they contain Stanley symmetric functions as a special case, 2) they share several analogous combinatorial properties and 3) they and their duals were conjecturally related by Jennifer Morse and Mark Shimozono to the geometry of the affine Grassmannian and ``affine Schubert polynomials,'' in a manner analogous to the relation of Schubert polynomials to the cohomology of the flag variety.  In \cite{lam08}, Lam indeed showed a geometric interpretation of affine Stanley symmetric functions as representing Schubert classes of the cohomology of the affine Grassmannian of $SL(n,\mathbb{C})$.  The dual homology representatives are $t=1$ specializations of the $k$-Schur functions of Lascoux, Lapointe and Morse \cite{llm03}, which implies a relationship between affine Stanley symmetric functions and Macdonald polynomials.

Given the geometric interpretation of affine Stanley symmetric functions, a natural question to ask is if there are symmetric polynomial representatives for the Schubert classes of the (co)homology of the affine Grassmannian corresponding to any Lie type.  In \cite{bott58}, Bott described the (co)homology of affine Grassmannian for all the classical types, but his descriptions lacked concrete realization.  In \cite{lss10}, Lam, Schilling and Shimozono found symmetric function representatives for the affine Grassmannian of the symplectic group, the type $C$ affine Stanley symmetric functions.  More recently, Lam \cite{lam09} explained the thesis that ``every affine Schubert class is a Schur-positive symmetric function.''  That is, given simple and simply-connected complex algebraic groups $G \subset G'$ with an inclusion $\iota: G \to G'$, there is a closed embedding of affine Grassmannians $\textrm{Gr}_G \to \textrm{Gr}_{G'}$ and the pushforward of a Schubert class of $H_*(\textrm{Gr}_G)$ is a nonnegative linear combination of Schubert classes in $H_*(\textrm{Gr}_{G'})$.  In the limit, $H_*(\textrm{Gr}_{SL(\infty, \mathbb{C})}) \cong \Lambda$, where $\Lambda$ is the Hopf algebra of symmetric functions, and the Schubert basis is represented by Schur functions.

Therefore, one could expect an interpretation of Schubert classes of the affine Grassmannian of any Lie type as (Schur-positive) symmetric functions.  However, it is not always possible to find an injective map $H_*(\Gr) \to \Lambda$ (for example, if $G = SO(2n)$), so allowances must be made, and the quotation marks above must remain.

\subsection{Current results}
In the following, we generalize the methods of \cite{lss10} to first identify the homology Schubert basis of the affine Grassmannian of type $B$ or $D$ with a subalgebra of the nilCoxeter algebra known as the \emph{affine Fomin-Stanley subalgebra}, where the Schubert basis is represented by noncommutative $k$-Schur functions.  Using a noncommutative ``Cauchy-type'' kernel, we produce symmetric functions that model the cohomology of the affine Grassmannian of the odd special orthogonal groups, \emph{type $B$ affine Schur functions} (which are generalized by type $B$ affine Stanley symmetric functions).  Additionally, we prove positivity statements for the dual functions, \emph{type $B$ $k$-Schur functions}.  In the type $D$ case, there is no embedding of (co)homology into symmetric functions, but we present candidate symmetric functions that are conjectured to approximate the (co)homology rings.  

The definition of noncommutative $k$-Schur functions or affine Stanley symmetric functions in any type depends in part on a subset of the Weyl group that we call the set of \emph{Pieri factors}; we prove a description of the set of Pieri factors that is ``type-free'' in that it works for all classical types.  Affine Stanley symmetric functions for all classical types have been programmed into the math software package Sage \cite{sage,sagecom}; the appendix contains some data for small rank cases.

\subsection{Acknowledgements} The author would like to thank Anne Schilling for comments and many helpful discussions, Thomas Lam and Mark Shimozono for fruitful conversations, and Nicolas Thi\'ery for sharing his expertise with Sage.

\section{Main Results}
\label{section: main results}

Let $G$ be a simple and simply-connected complex algebraic group. Given such a group, we can associate a Cartan datum $(I,A)$ and Weyl group $W$ (see, for example, \cite{kum02}).  Let $K$ be a maximal compact subgroup of $G$, and let $T$ be a maximal torus in $K$.

Let $\mathbb{F} = \mathbb{C}((t))$ and $\mathbb{O} = \mathbb{C}[[t]]$.  The \emph{affine Grassmannian} may be given by $\Gr := G(\mathbb{F})/G(\mathbb{O})$.  $\Gr$ can be decomposed into \emph{Schubert cells} $\Omega_w = \mathcal{B}wG(\mathbb{O}) \subset G(\mathbb{F})/G(\mathbb{O})$, where $\mathcal{B}$ denotes the Iwahori subgroup and $w \in \Wgrass$, the set of Grassmannian elements in the associated affine Weyl group.  The Schubert varieties, denoted $X_w$, are the closures of $\Omega_w$, and we have $\Gr = \sqcup \Omega_w = \cup X_w$, for $w \in \Wgrass$.  The homology $H_*(\Gr)$ and cohomology $H^*(\Gr)$ of the affine Grassmannian have corresponding Schubert bases, $\{ \xi_w\}$ and $\{\xi^w\}$, respectively, also indexed by Grassmannian elements.  It is well-known that $\Gr$ is homotopy-equivalent to the space $\Omega K$ of based loops in $K$ (due to Quillen, see \cite[\textsection 8]{ps86} or \cite{mit88}).  The group structure of $\Omega K$ gives $H_*(\Gr)$ and $H^*(\Gr)$ the structure of dual Hopf algebras over $\mathbb{Z}$.

We study the Lie types $B$ and $D$ cases.  The complex special orthogonal groups $G = SO(n,\mathbb{C})$ are not simply-connected, but we may consider $G = \textrm{Spin}(n,\mathbb{C})$, and $K = \textrm{Spin}(n)$ (note also that the loop space $\Omega \textrm{Spin}(n) \cong \Omega_0 SO(n)$, the connected component of the identity).  The corresponding Weyl groups will be denoted $\WB$ (for $\Spin(2n+1)$) and $\WD$ (for $\Spin(2n)$).  We will also denote the affine Grassmannian of each type by $\GrB$ and $\GrD$, respectively.

\subsection{Type free results and definitions}  
\label{subsection: type free results}

Given $\Waf$, an affine Weyl group of classical type, let $\fincoweight{1},\ldots, \fincoweight{n}$ be the corresponding finite fundamental coweights.  We may identify elements of $\Waf$ with the set of alcoves in the weight space of the associated finite Lie algebra.  Let $\mathcal{O}$ be the orbit of $\nu(\fincoweight{1})$ under the usual action of the finite Weyl group, where $\nu$ is the usual map from the Cartan subalgebra to its dual.  We then define the set of \emph{Pieri factors} to be the Bruhat order ideal of $\Waf$ generated by the alcoves corresponding to translations of the identity alcove by elements of $\mathcal{O}$.  These Pieri factors will lead to the definition of affine Stanley symmetric functions in each type, and we denote them by $\Z$ (in order to specify type, we will use the notation $\ZB$, $\ZD$, etc.).  Furthermore, let the length $i$ elements of $\Z$ (resp. $\ZB, \ZD$) be denoted by $\Z_i$ (resp. $\ZB_i, \ZD_i$).

The type-free Pieri factors described above match with the corresponding set of affine Weyl group elements given in type $A$ (\cite[Definition 6.2]{lam08}), type $C$ (\cite[\textsection 1.5]{lss10}), and types $B$ and $D$ below (Definitions \ref{def: type B pieri factors} and \ref{def: type D pieri factors}).  See Proposition \ref{prop: type-free pieri factors} for a proof of this fact.

We will also need the following definitions.  Define `` $\prec$'' on $\Iaf$ in type $B$ by $0,1 \prec 2 \prec 3 \prec \cdots \prec n$, i.e., 0 and 1 are incomparable.  Similarly, define ``$\prec$'' on $\Iaf$ in type $D$ by $0,1, \prec 2 \prec 3 \prec \cdots \prec n-2 \prec n-1,n$.  In each case, this is the ordering suggested by the Dynkin diagram of affine type $B$ or $D$, respectively.  Define an \emph{interval} $[m,M]$ to be the set $\{ j \in \Iaf : j \nprec m \textrm{ and } j \nsucc M \} $.  Note that this implies that any interval either includes both 0 and 1, or includes neither (and in type $D$, any interval includes both $n-1$ and $n$, or neither).

Given an element $w \in \Z$, define the \emph{pre-support} of $w$, $\mathrm{supp}(w)$, to be the subset of $\Iaf$ consisting of the indices that appear in a reduced word for $w$.  Define the \emph{support} of $w$, $\Supp(w)$, to be the smallest union of intervals containing $\mathrm{supp}(w)$.  By the Coxeter relations for $\Waf$, these are independent of choice of reduced word and therefore well-defined.

The \emph{complement} of $\Supp(w)$ is $\Iaf \setminus \Supp(w)$.  When the complement of $\Supp(w)$ is written as a minimal number of disjoint intervals, we say those intervals are the \emph{components} of the complement of $\Supp(w)$.   Let $c(w) = $ the number of components of the support of $w$, and let $\stat(w) = $ the number of components of the complement of the support of $w$.

\begin{example} 
In type $B$, suppose $n = 7$ and $w = 3621 \in \ZB$.  Then $\mathrm{supp}(w) = \{ 1, 2, 3, 6 \}$, $\Supp(w) = [0,3] \cup \{ 6 \}$ and the complement of $\Supp(w)$ is $[4,5] \cup \{7 \}$ so $\stat(w) = 2$.
\end{example}

Let $\ell(w)$ be the length function on Weyl group elements.
 
\begin{definition} We define affine Stanley symmetric functions for any type by
\label{def: affine Stanley symmetric functions}
$$
\ASSF{w}[y] = \sum_{(v^1, v^2, \ldots)} \prod_i 2^{\genericstat(v^i)-1} y_i^{\ell(v_i)}
$$
where the sum runs over the factorizations $v^1v^2 \cdots = w$ of $w$ such that $v^i \in \Z$ and $\ell(v^1) + \ell(v^2) + \cdots = \ell(w)$, and $\genericstat$ is a statistic on Pieri factors that is type-specific.  For type $A$, $\genericstat(w) = 1$ for all $w \in \Z^A$.  For type $C$, $\genericstat(w) = c(w)$, and for types $B$ and $D$, $\genericstat(w) = \stat(w)$.
\end{definition}

We note that this definition of affine Stanley symmetric functions matches with those of \cite{lam05, lam08, lam09, lss10}.  Affine Stanley symmetric functions for specific types will be denoted by a superscript (e.g., $\ASSFB{w}$).

\subsection{Type $B$ main results}

In terms of reduced words, the type $B$ Pieri factors are given below.  Definition \ref{def: type B pieri factors} is used to prove the type-free Pieri factor formulation above.

\begin{definition}
\label{def: type B pieri factors}
The type $B$ Pieri factors are generated by the length-maximal elements with reduced words $s_0 s_2 \cdots s_n \cdots s_2 s_0$, $s_1 s_2 \cdots s_n \cdots s_2 s_1$, $s_2 s_3 \cdots s_n \cdots s_2 s_1 s_0$, and all cyclic rotations of the latter such that $s_0$ and $s_1$ remain adjacent (for example, the element with reduced word  $s_1 s_2 \cdots s_n \cdots s_2 s_0$ is \emph{not} a generator).  By Proposition \ref{prop: type-free pieri factors}, this matches with the above type-free definition of Pieri factors.
\end{definition}

Let $\Lambda$ be the ring of symmetric functions over $\mathbb{Q}$, and let $P_i$ and $Q_i$ denote the Schur $P-$ and $Q-$functions with a single part.  Let $\Gamma_* = \mathbb{Q}[Q_1,Q_2, \ldots]$, and let $\Gamma^* = \mathbb{Q}[P_1,P_2, \ldots]$; then $\Gamma_*$ and $\Gamma^*$ are dual Hopf algebras under the pairing $[\cdot, \cdot]$ given in \cite{mac95} (in fact, $\Gamma^* = \Gamma_*$, but it will be convenient to distinguish them as we begin considering $\mathbb{Z}$-algebras).  Let $\homB = \mathbb{Z}[Q_1,Q_2, \ldots, Q_{n-1},2Q_n, \ldots, 2Q_{2n-1}] \subset \Gamma_*$ be a Hopf algebra over $\mathbb{Z}$, and let $\cohomB$ be the dual quotient $\mathbb{Z}$-Hopf algebra embedded in $\Gamma^*$.

The finite Weyl group $B_n$ sits inside $\WB$ as the group generated by simple reflections $s_1, \ldots, s_n$.  We fix a set of minimal-length coset representatives of $\WB/B_n$, which we refer to as \emph{Grassmannian} (or \emph{0-Grassmannian}) elements, and denote them by $\Bgrass$.

\begin{theorem}
\label{theorem: ASSF form positive basis}
The functions $\ASSFB{w}, w \in \Bgrass$ form a basis of $\cohomB$ such that all product and coproduct structure constants are positive, and all $\ASSFB{w}$ with $w \in \WB$ are positive in this basis.
\end{theorem}

Let the Grassmannian elements $\rho_i \in \WB$ be given by:
\begin{equation}
\label{eq: type B rho}
\rho_i = \begin{cases}
s_0 & i =1\\ 
s_i \cdots s_3 s_2 s_0 & 2 \leq i \leq n \\
s_{2n-i} s_{2n-i+1} \cdots s_{n-1} s_n s_{n-1} \cdots s_2 s_0 & n \leq i \leq 2n-2 \\
s_0 s_2 s_3 \cdots s_{n-1} s_n s_{n-1} \cdots s_3 s_2 s_0 & i = 2n-1. 
\end{cases}
\end{equation}

\begin{theorem}
\label{theorem: type B cohomology isomorphism}
There are dual Hopf algebra isomorphisms 
$$
 \Phi: \homB \to H_*(\GrB) \qquad \textrm{ and } \qquad  \Psi : H^*(\GrB) \to \cohomB
 $$ 
 such that $\Phi(2^{\chi(i \geq n)} Q_i) = \xi_{\rho_i}$ for $1 \leq i \leq 2n-1$, and $\Psi(\xi^w) = \ASSFB{w}$ for $w \in \Bgrass$.  
\end{theorem}
Furthermore, we will show (Propositions \ref{prop: type B k-schur are schur P positive} and \ref{prop: type B k-schur are k-schur positive}) that the embeddings of symmetric functions $\homB \to \Gamma_{(n+1)}^{B_n}$ and $\homB \cong H_*(\Spin(2n+1)) \hookrightarrow H_*(\Omega SU(2n+1)) \cong \mathbb{Z}[h_1, h_2, \ldots, h_{2n}]$ induced by the above isomorphisms agree with the natural embeddings of symmetric functions.  The elements of the basis of $\homB$ dual to $\{\ASSFB{w}\}$ we call type $B$ $k$-Schur functions, and denote by $\{\kSB{w}\}$.

\begin{theorem}
\label{theorem: type B pieri rule}
Given $w \in \Bgrass$, we have in $H_*(\GrB)$:
$$
\xi_{\rho_i} \xi_w = \sum_{v \in \ZB_i} 2^{\stat(v) - \chi(i < n)} \xi_{vw},
$$
where the sum is over $v$ such that $vw \in \Bgrass$ and $\ell(vw) = \ell(v) + \ell(w)$.
\end{theorem}

The proofs of the above theorems may be found in \textsection \ref{section: proof of main theorems}.

\subsection{Type $D$ main results}
In type $D$, the situation is not as favorable.  As noted in \cite{lam09}, there may not be a Hopf inclusion $H_*(\textrm{Gr}_D) \hookrightarrow \Lambda$.  However, it may be possible to approximate $H_*(\textrm{Gr}_D)$ with symmetric functions by using a slightly non-injective map.  We conjecture dual symmetric function algebras that approximate the (co)homology of $\GrD$.

\begin{definition}
\label{def: type D pieri factors}
Type $D$ Pieri factors are generated by the following affine Weyl group elements: $s_0 s_2 \cdots s_{n-2} s_n s_{n-1} s_{n-2} \cdots s_2 s_0$, $s_0 s_1 s_2 \cdots s_{n-2} s_n s_{n-1} s_{n-2} \cdots s_2$, all cyclic rotations of the latter such that $s_0,s_1$ remain adjacent and $s_{n-1},s_n$ remain adjacent, and all images of these words under Dynkin diagram automorphisms.
\end{definition}

Let the Grassmannian elements $\rho_i \in \WD$ be given by:
\begin{equation}
\label{eq: type D rho}
\rho_i = 
\left \{ \begin{array}{cc}
s_0 & i =1\\ 
s_i \cdots s_3 s_2 s_0 & 2 \leq i < n-1 \\
 s_{2n-1-i} s_{2n-i} \cdots s_{n-2} s_n s_{n-1} s_{n-2} \cdots s_2 s_0 & n-1 < i < 2n-2 \\ 
 s_0 s_2 s_3 \cdots s_{n-2} s_n s_{n-1}s_{n-2} \cdots s_3 s_2 s_0 & i = 2n-2. 
 \end{array} \right. 
\end{equation}
and let $\rho_{n-1}^{(1)} = s_ns_{n-2}\cdots s_2 s_0$ and $\rho_{n-1}^{(2)} = s_{n-1}s_{n-2} \cdots s_2 s_0$.
Let 
$$
\homD =  \mathbb{Z}[Q_1, \ldots, Q_{n-1}, 2Q_n, \ldots, 2Q_{2n-2}],
$$
and let $\cohomD$ be the dual quotient $\mathbb{Z}$-Hopf algebra embedded in $\Gamma^*$.

It is impossible to find a surjective map from $\homD$ onto $H_*(\GrD)$ -- as stated in \cite{lam09}, $H_*(\GrD)$ may have a primitive subspace of dimension 2 in a given degree, whereas $\Lambda$ has primitive spaces of dimension 1 in all degrees.  Thus the pushforward  $i_*: H_*(\Omega \Spin(2n)) \to H_*(\Omega SU(2n))$ must have a nontrivial kernel.  
\begin{conj}
\label{conj: type D conjecture}
The kernel $\gamma$ of $i_*: H_*(\Omega \Spin(2n)) \to H_*(\Omega SU(2n))$ is generated by  $\xi_{\rho_{n-1}^{(2)}} - \xi_{\rho_{n-1}^{(1)}}$, and $H_*(\Omega \Spin(2n))/\gamma$ is isomorphic to $\homD$; under the dual isomorphism, affine Stanley symmetric functions represent cohomology Schubert classes.  Furthermore, the inclusion $\homD \cong H_*(\Omega \Spin(2n))/\gamma  \to H_*(\Omega SU(2n)) \cong \mathbb{Z}[h_1,\ldots,h_{2n-1}]$ corresponds to the natural inclusion of symmetric functions.
\end{conj}

Although we cannot describe the (co)homology explicitly via symmetric functions, our description of Pieri factors, however, is enough to present a Pieri rule for type $D$ homology.  Unfortunately, for some cases (i.e., $i = n-1$), the rule is still complicated.
\begin{theorem}
\label{theorem: type D pieri rule}
Given $w \in \Dgrass$ and $i \neq n-1$, we have in $H_*(\GrD)$:
$$
\xi_{\rho_i} \xi_w = \sum_{v \in \ZD_i} 2^{\stat(w) - \chi(i < n)} \xi_{vw},
$$
where the sum is over $v$ such that $vw \in \Dgrass$ and $\ell(vw) = \ell(v) + \ell(w)$.  

If $i = n-1$, we have
$$
\xi_{\rho_{n-1}^{(1)}} \xi_w = \sum_{v \in \ZD_{n-1}} c_v \xi_{vw}
$$
where $c_v$ is the coefficient of $A_v$ in $\specialgenD{n-1} + \epsilon$ (see Section \ref{subsection: type D special generators} for definitions).  A similar formula holds for $\xi_{\rho_{n-1}^{(2)}} \xi_w$.
\end{theorem}

The proofs of the above theorems are contained in Section \ref{section: proof of main theorems}.  Some proofs of supporting lemmas are very lengthy and similar enough to those contained in \cite{lss10} that we refer the reader to \cite{lss10} or \cite{pon10} for details.

\subsection{Future directions} Many natural questions remain; a small sample includes the following.
\begin{itemize}
\item A proof of Conjecture \ref{conj: type D conjecture}.  This most likely will involve a hard spectral sequence computation, and techniques similar to the type $B$ case.
\item   Although a type-free description of Pieri factors exists, the current proof is verified on a type-specific basis -- it would be desirable to have a type-free proof of our description.  Such a type-free proof would almost certainly require a type-free description of the statistic $\genericstat(w)$, which might lead easily to Pieri factors for the exceptional types.
\item The $\tilde{Q}-$functions studied extensively by Pragacz (see, for example, \cite{pr97, lp00}) may be a more natural symmetric function model for homology; it would be interesting to see how they relate.
\end{itemize}

\section{Background}
\label{section: background}

In this section, we present the necessary background, mostly following the conventions and notation of \cite{mac95} and \cite{kac90}.

\subsection{Symmetric functions}
Let $\Lambda$ be the ring of symmetric functions, and let $\parts$ be the set of partitions.  Let $h_\lambda$ the complete homogeneous symmetric functions, $p_\lambda$ the power sum symmetric functions, and $m_\lambda$ the monomial symmetric functions for $\lambda \in \parts$.

Schur's $P-$ and $Q-$ functions are symmetric functions that arose in the study of projective representations of the symmetric group, where they play the role of Schur functions in linear representations of the symmetric group.  They may be defined in several ways; we present a combinatorial definition as sums over \emph{shifted tableaux}, due to Stembridge \cite{ste89}.  

Let $\lambda \in \parts$ be a strict partition (that is, a partition with all parts distinct), and define an alphabet $\mathcal{A} = \{ \bar{1},1,\bar{2},2,\ldots\}$ and partial ordering $\bar{1} < 1 < \bar{2} < 2 < \cdots$.  Then a marked, shifted tableau of shape $\lambda$ is a diagram of $\lambda$ where row $i$ is shifted by $i-1$ spaces, and all boxes of $\lambda$ are filled with letters from $\mathcal{A}$ such that i) labels weakly increase along rows and columns, ii) columns have no repeated unbarred letters, and iii) rows have no repeated barred letters.

For a shifted marked tableau $T$, we may then define $x^T = x_1^{c_1}x_2^{c_2}\cdots$, where $c_i = $ the number of $i$'s in $T$ (both barred and unbarred).  Then
\begin{equation}
Q_\lambda = \sum_{T} x^T,
\end{equation}
the sum over all shifted marked tableaux of shape $\lambda$.  Schur's $P-$functions are scalar multiples of the $Q_\lambda$; $P_\lambda = 2^{-\ell(\lambda)}Q_\lambda$, where $\ell(\lambda)$ is the number of nonzero parts of $\lambda$.  By the definition given above, Schur $P-$ and $Q-$ functions are defined only for strict partitions.  We denote the set of strict partitions by $\mathcal{SP}$.  The $P_\lambda$ are a basis for $\Gamma^*$ defined above, and the $Q_\lambda$ are a basis for $\Gamma_*$.

\begin{example}
A shifted tableau of shape $(6,4,3)$ is given below.  In the formula for Schur $Q$-functions, this tableau would correspond to a monomial $x_1x_2^2x_3^5x_4^2x_5^2x_7$.
\begin{center}
\young(1\bartwo\barthree334,:2\barthree45,::357)
\end{center}
\end{example}

It is well known that the ring of symmetric functions has a Hopf algebra structure.  On $\Gamma_*$, the coproduct is given by
\begin{equation}
\Delta(Q_r) = 1 \otimes Q_r + Q_r \otimes 1 + \sum_{1 \leq s < r} Q_s \otimes Q_{r-s}.
\end{equation}

Furthermore, the $Q_i$ satisfy only the relations:
\begin{equation}
\label{Q relations}
Q_{i}^2 = 2(Q_{i-1}Q_{i+1} - Q_{i-2}Q_{i+2} + \cdots \pm Q_{0}Q_{2i} )
\end{equation}
where we let $Q_0 = 1$ \cite[III.8.2']{mac95}.\\

Define the Hall-Littlewood scalar product, a pairing $[\cdot,\cdot] : \Gamma_* \times \Gamma^* \to \mathbb{Z}$, by $[Q_\lambda, P_\mu] = \delta_{\lambda \mu}$ for $\lambda, \mu \in \mathcal{SP}$.

The pairing $[\cdot, \cdot]$ has reproducing kernel
\begin{eqnarray}
\Omega_{-1} &:=& \prod_{i,j \geq 1} \frac{1 + x_iy_j}{1-x_iy_j}\\
&=& \sum_{\lambda \in \mathcal{SP}} 2^{-\ell(\lambda)}Q_\lambda[X]Q_\lambda[Y]\\
&=&\sum_{\lambda \in \mathcal{P}} Q_{\lambda_1}[X]Q_{\lambda_2}[X]\cdots m_{\lambda}[Y]
\end{eqnarray}
where the second equality is by \cite[III.8.13]{mac95} and the third is by setting $t = -1$ in \cite[III.4.2]{mac95}.  Following Macdonald, we denote $q_\lambda = Q_{\lambda_1}Q_{\lambda_2}\cdots$.

\subsection{Weyl groups}For more background on Weyl groups, see \cite{kac90}, \cite{car05}, \cite{hum90}, \cite{bb05}.  We will assume all Lie algebras are non-twisted.

Let $(\Iaf,\Aaf)$ denote a Cartan datum of affine type, and denote the corresponding finite type Cartan datum by $(I,A)$.  The affine Weyl group $\Waf$ corresponding to $(\Iaf, \Aaf)$ is given by generators $s_i$ for $i \in \Iaf$, and relations $s_i^2 = 1$,
\begin{equation}
\label{braid relations}
(s_is_j)^{m(i,j)} = 1 \qquad \textrm{ for } i \neq j 
\end{equation}
 where $m(i,j) = 2,3,4,6,$ or $\infty$ as $a_{ij}a_{ji}$ equals $0,1,2,3,$ or $\geq 4$.  The associated finite Weyl group $W$ has the same relations, but with generators $s_i, i \in I$.

Given any element $w$ of $\Waf$ or $W$, there are a number of words in the generators $s_i$ for $w$, all of which are connected via the braid relations (\ref{braid relations}).  The \emph{length function} $\ell : \Waf \to \mathbb{Z}_{\geq 0}$ is given by $\ell(w) = k$ if $k$ is minimal such that $s_{i_1} s_{i_2} \cdots s_{i_k}$ is a word for $w$.  If $\ell(w) = k$, we call an expression $w = s_{i_1} \cdots s_{i_k}$ a \emph{reduced expression}, and call $i_1 \cdots i_k$ a \emph{reduced word} for $w$.  We denote the set of all reduced words for $w$ by $R(w)$.  Elements of $\Waf$ or $W$ that are conjugates of the $s_i$ are called reflections.  

The Bruhat order on $\Waf$ or $W$ is a partial ordering given by $v \leq w$ if some (equivalently, every) reduced word for $v$ is a subword of a reduced word for $w$.  We let $\lessdot$ denote the covering relation of Bruhat order, so that $v \lessdot w$ if $v \leq w$ and $\ell(v) + 1 = \ell(w)$.

Given a subset $J$ of $\Iaf$, we will define the parabolic subgroup $(\Waf)_J \subset \Waf$ as the subgroup generated by $\{s_i \mid i \in J\}$, and denote by $\Waf^J$ a set of minimal length coset representatives for $\Waf/(\Waf)_J$.\\

\subsection{(Co)roots and (co)weights}

We will let $\gaf$ be the affine Kac-Moody algebra associated to a Cartan datum $(\Iaf, \Aaf)$, and let $\g$ be the associated finite Lie algebra with Cartan datum $(I = \Iaf \setminus \{0\},A = (\Aaf)_{i,j = 1}^n)$ (see \cite{kac90} for details).  The corresponding Cartan subalgebras are denoted $\haf$ and $\h$, respectively.  Given an affine Cartan datum, we have the set of simple roots $\Piaf = \{\alpha_{0}, \alpha_{1}, \ldots \alpha_{n} \} \subset \haf^*$ and simple coroots $\Piaf^\vee = \{ \coroot{0}, \coroot{1}, \ldots, \coroot{n} \} \subset \haf$ such that $\langle \coroot{i}, \alpha_j,  \rangle := \alpha_j(\coroot{i}) = a_{ij}$.  Since $\Aaf$ is of corank 1, there is a unique positive integer vector $a = (a_0, a_1, \ldots, a_n)$ whose entries have no common factor such that $\Aaf \,a = 0$.  We let $\delta = a_0\alpha_0 + a_1\alpha_1 + \cdots + a_n \alpha_n = a_0\alpha_0 + \theta \in \haf^*$ be the null root.  Correspondingly, there is a vector $a^\vee = (a_0^\vee, \ldots, a_n^\vee)$ such that $a^\vee \Aaf = 0$; we let $K = a_0^\vee \coroot{0} + \cdots + a_n^\vee \coroot{n} \in \haf$ be the canonical central element.  There is a basis of $\haf$ given by $\{\coroot{0}, \ldots, \coroot{n},d\}$, where $d$ is the scaling element such that $\langle \alpha_i, d \rangle = \delta_{0i}$.  We also have the fundamental weights $\Lambda_0, \ldots, \Lambda_n \in \haf^*$ such that $\{\Lambda_0, \ldots, \Lambda_n, \delta \}$ is dual to the above basis of $\haf$.  The affine root lattice is denoted by $\Qhat = \bigoplus_{i \in \Iaf} \mathbb{Z}\alpha_i$ and the affine coroot lattice by $\Qhat^\vee = \bigoplus_{i \in \Iaf} \mathbb{Z}\coroot{i} \oplus \mathbb{Z}d$.  The affine weight lattice is $\Phat = \bigoplus_{i \in \Iaf} \mathbb{Z}\Lambda_i \oplus \mathbb{Z}\delta$, and the affine coweight lattice $\Phat^\vee = \bigoplus_{i \in \Iaf} \mathbb{Z} \coweight{i}$ is generated by the fundamental coweights, which are dual to the simple roots.

The finite coroot lattice is $Q^\vee = \bigoplus_{i \in I} \mathbb{Z} \coroot{i} \subset \Qaf^\vee$.  The finite root lattice $Q$ is a quotient of $\Qaf$, but we will identify it with a sublattice $Q = \bigoplus_{i \in I} \mathbb{Z}\alpha_i \subset \Qaf$.  The finite fundamental weights can be embedded in $\Paf$ by $\omega_i = \Lambda_i - \langle \Lambda_i, K\rangle \Lambda_0$, for $i \in I$, and as in the case of the finite root lattice, we identify the finite weight lattice $P$ with a sublattice $P = \bigoplus_{i \in I} \mathbb{Z} \omega_i \subset \Paf$.  The finite coweight lattice is denoted by $P^\vee = \bigoplus_{i \in I} \mathbb{Z}\fincoweight{i}$, where $\alpha_i(\fincoweight{j}) = \delta_{ij}$..

\subsection{Geometric representation of the affine Weyl group}

We have an action of $\Waf$ on $\haf$ given by $s_i(\mu) = \lambda - \langle \mu, \alpha_{i} \rangle \alpha_i^\vee$, and $w(K) = K$ for $w \in \Waf$.  More generally, for any real root $\alpha^\vee \in \haf$, we have the element $s_\alpha \in \Waf$ which acts by $s_\alpha(\mu) = \mu - \langle \mu, \alpha \rangle \alpha^\vee$.  Also for elements $\alpha \in \h^*$, we have the ``translation'' endomorphism of $\haf^*$, $t_\alpha$, given by
\begin{equation}
t_\alpha(\lambda) = \lambda + \langle \lambda, K \rangle \alpha - ((\lambda \mid \alpha) + \frac{1}{2} |\alpha|^2 \langle \lambda, K \rangle ) \delta.
\end{equation}
It is not hard to show that $t_\alpha t_\beta = t_{\alpha + \beta}$, and also that $t_{w(\alpha)} = w t_{\alpha} w^{-1}$ for $w \in W$.  We let $M = \nu(Q^\vee)$, and let the abelian group generated by $\{t_\alpha \mid \alpha \in M\}$ be denoted $T_M$.  The affine Weyl group can be presented as $\Waf = W \ltimes T_M$.

\subsection{Alcoves} Let $\h^*_\mathbb{R}$ be the $\mathbb{R}$-linear span of the finite simple roots, and let $\h_\mathbb{R}$ be the $\mathbb{R}$-linear span of the finite simple coroots.  Then let $\haf^* \otimes \mathbb{R} = \h^*_\mathbb{R} + \mathbb{R}K + \mathbb{R}d$, and $\haf \otimes \mathbb{R} = \h_\mathbb{R} + \mathbb{R}\Lambda_0 + \mathbb{R}\delta$.  Let $(\haf^*)_s = \{ \lambda \in \haf^* \otimes \mathbb{R} \mid \langle \lambda, K \rangle = s\}$.  The hyperplanes $(\haf^*)_s$ are invariant under the $\Waf$ action described above, and the action of $\Waf$ on $(\haf^*)_0$ is faithful.  Furthermore, the action of $\Waf$ on $(\haf^*)_1 / \mathbb{R}\delta$ is also faithful.  We can identify $9\haf^*)_1 / \mathbb{R}\delta$ with $\h^*_\mathbb{R}$ by projection, thereby identifying $\Waf$ with a group of affine transformations of $\h^*_\mathbb{R}$.  Under this isomorphism, $t_\alpha$ corresponds to translation by $\alpha$, for $\alpha \in M$.  Define the hyperplanes $H_{\alpha,k}$ in $\h^*_\mathbb{R}$ by $H_{\alpha,k} = \{ x \in \h^*_\mathbb{R} \mid (\alpha \mid x) = k\}$, and define the fundamental alcove $\funalc$ as the domain bounded by $\{ H_{\alpha_i, 0} \mid i = 1,\ldots, n\} \cup \{ H_{\theta,1} \}$.  The fundamental alcove is a fundamental domain for the action of $\Waf$ on $\h^*_\mathbb{R}$.  The images of $\funalc$ are the alcoves, and are in bijection with the elements of $\Waf$ via $w \leftrightarrow w\funalc$.

\subsection{The extended affine Weyl group} Just as we can see the affine Weyl group as $\Waf = W \ltimes T_M$, we can view the extended affine Weyl group as $\Wafext = W \ltimes T_{\widetilde{M}}$, where $\widetilde{M} = \nu(P^\vee)$ and the action of translations is as above.  Let $C$ be the dominant Weyl chamber, $C = \{ \lambda \in \Phat \otimes_\mathbb{Z} \mathbb{R} \mid \langle \coroot{i}, \lambda \rangle \geq 0 \textrm{ for all } i \in \Iaf\}$.  If we let $\Sigma$ be the subgroup of $\Wafext$ stabilizing $C$, then we can write $\Wafext = \Sigma \ltimes \Waf$, where $\tau s_i \tau^{-1} = s_{\tau(i)}$ for $\tau \in \Sigma$.  The group $\Sigma$ is the finite group of Dynkin diagram automorphisms -- permutations of the nodes of the Dynkin diagram that preserve the graph structure.  Elements of $\Sigma$ permute the simple roots; i.e., if $\tau(i) = j$, then $\tau(\alpha_i) = \alpha_j$.  

One can view the extended affine Weyl group as acting on $|\Sigma|$ copies of $\h^*_\mathbb{R}$.  Label the wall of $\funalc$ formed by $H_{\alpha_i,0}$ by $i$ and the wall formed by $H_{\theta,1}$ by $0$.  Label the walls of all other alcoves so that the labeling is $\Waf$-equivariant.  Then elements of $\Sigma$ correspond to permuting the labels on all alcoves, which we can view as transitioning between different copies of $\h^*_\mathbb{R}$.

\section{The affine Grassmannian}

This section imitates \cite[\textsection 2]{lam08} and \cite[\textsection 4]{lss10}, and introduces the necessary algebraic and geometric background to prove our main theorems.  See \cite{pet97, lam08, lss10} for more details.

\subsection{The nilCoxeter algebra} Given a Cartan datum $(\Iaf, \Aaf)$ of affine type, we can define the affine nilCoxeter algebra $\mathbb{A}_0$ as the associative $\mathbb{Z}$-algebra with generators $A_i$ for $i \in \Iaf$ and relations $A_i^2 = 0$ for all $i \in \Iaf$, and $(A_iA_j)^{m(i,j)} = 1$, where $m(i,j)$ is as in the definition of the affine Weyl group.  Since the braid relations are the same as for the affine Weyl group, given any $w \in \Waf$ and $i_1 i_2 \cdots i_\ell \in R(w)$, the element $A_w = A_{i_1} A_{i_2} \cdots A_{i_\ell} \in \mathbb{A}_0$ is well-defined.  We write $\mathbb{A}_0^B$ (resp. $\mathbb{A}_0^D$) for the nilCoxeter algebra of type $B$ (resp. $D$) when we want to refer to a specific type.

\subsection{The nilHecke algebra} Let $S = $Sym($\Paf$), the symmetric algebra generated by the affine weight lattice.  Peterson's affine nilHecke algebra $\mathbb{A}$ is the associative $\mathbb{Z}$-algebra generated by $S$ and the nilCoxeter algebra $\mathbb{A}_0$, with commutation relation
\begin{equation}
A_i \lambda = (s_i \cdot \lambda)A_i + \langle \coroot{i}, \lambda \rangle 1 \qquad \textrm{ for } i \in \Iaf \textrm{ and } \lambda \in \Paf.
\end{equation}

There is a coproduct on $\mathbb{A}$ given by 
\begin{align}
\Delta(A_i) &= A_i \otimes 1 + 1 \otimes A_i - A_i \otimes \alpha_i A_i\\
\Delta(s) &= s \otimes 1
\end{align}

Let $\phi_0: S \to \mathbb{Z}$ be given by $\phi_0(s) = $ evaluation of $s$ at 0.  By abuse of notation, let $\phi_0:\mathbb{A} \to \mathbb{A}_0$ be the map given by
\begin{equation}
     \phi_0 : \sum_w a_w A_w \longrightarrow \sum_w \phi_0(a_w)A_w.
\end{equation}

Define the affine Fomin-Stanley subalgebra 
\begin{equation}
\mathbb{B} = \{ a \in \mathbb{A}_0 \mid \phi_0(as) = \phi_0(s)a \textrm{ for all } s \in S \}.
\end{equation}
We can define the restriction map $\phi_0^{(2)} : \mathbb{A} \otimes_S \mathbb{A} \to \mathbb{A}_0 \otimes_{\mathbb{Z}} \mathbb{A}_0$ by
\begin{equation}
\phi_0^{(2)} \left ( \sum_{w,v \in \Waf }a_{w,v} A_w \otimes A_v \right)  = \sum_{w,v \in \Waf} \phi_0(a_{w,v})A_w \otimes A_v
\end{equation}
for $a_{w,v} \in S$.  $\mathbb{B}$ inherits the coproduct from $\mathbb{A}$ via $\phi_0^{(2)} \circ \Delta$.

Following Peterson's work \cite{pet97} as described in \cite{lss10}, there is an injective ring homomorphism $j_0: H_*(\Gr) \to \mathbb{A}_0$.  Given the coproduct inherited from $\mathbb{A}$, $j_0$ restricts to a Hopf algebra isomorphism $H_*(\Gr) \cong \mathbb{B}$.  The following theorem may be found in \cite{lam08}, Proposition 5.4 and Theorem 5.5, or \cite{lss10}, Theorem 4.6.
\begin{theorem}[\cite{pet97}, \cite{lam08},\cite{lss10}]
\label{theorem: j_0 is isomorphism}
There exists a Hopf algebra isomorphism 
$$
j_0: H_*(Gr_G) \to \mathbb{B}
$$ 
such that for all $w \in \Waf^0$, $j_0(\xi_w)$ is the unique element of $\mathbb{B} \cap (A_w + \sum_{u \in \Waf \setminus \Waf^0} \mathbb{Z}A_u)$.
\end{theorem}
We also have the following:
\begin{theorem}[\cite{ls07}, Theorem 6.3]
\label{theorem: homology pieri rule}
Let $j_w^u$ be defined by $j_0(\xi_w) = \sum_{\stackrel{u \in \Waf}{\ell(u) = \ell(w)}} j_w^u A_u$ for $w \in \Wgrass, u \in \Waf$.  Then for $x,z \in \Wgrass$, $\xi_x \xi_z = \sum_{y} j_x^y \xi_{yz}$ where the sum is over $y \in \Waf$ such that $yz \in \Wgrass$ and $\ell(yz) = \ell(y) + \ell(z)$.
\end{theorem}

$\mathbb{B}$ is a commutative algebra, and has a basis given by $j_0(\xi_w)$, for $w \in \Wgrass$.  Define $\specialgen{w} := j_0(\xi_w), \textrm{ for } w \in \Wgrass$.  Lemma \ref{lss10 lemma 4.7} helps to compute the elements $\specialgen{w}$.  

Suppose $w \gtrdot v$ in $\Waf$.  Then $s = v^{-1}w$ is a reflection, and we can write $s = u s_i u^{-1}$ for some simple reflection $s_i$, $i \in \Iaf$.  Let $u$ be shortest such that $\alpha = u (\alpha_i)$ is a positive real root.  Denote this root $\alpha$ by $\alpha_{vw}$ and its associated coroot, $u( \alpha_i^\vee)$, by $\asscovercoroot$.  We call $\asscovercoroot$ the associated cover coroot of $v$ and $w$.

\begin{lemma}[\cite{lss10},Lemma 4.7].
\label{lss10 lemma 4.7} 
Let $a = \sum_{w \in W_{\mathrm{af}}} c_w A_w \in \mathbb{A}_0$ with $c_w \in \mathbb{Z}$.  Then $a \in \mathbb{B}$ if and only if $\sum_{w \gtrdot v} c_w \asscovercoroot \in \mathbb{Z}K$ for all $v \in W_{\mathrm{af}}$.
\end{lemma}

\section{Special orthogonal groups}

In order to prove our theorems for the special orthogonal groups, we need to develop indexing sets and Cauchy-type reproducing kernels for each case.

\subsection{Type $B$ reproducing kernel} The first of the following sets of partitions was defined in \cite{ee98} (see also \cite{bm08}).
\begin{definition}
The set of type $B_n$ affine partitions is $\Bpar = \{ \lambda \mid \lambda_1 \leq 2n-1 \textrm{ and } \lambda \textrm{ has distinct parts of size smaller than } n\}$.  We also define $\Bodd{k}$ to be the set of $k$-bounded partitions with odd parts.  We will denote the set of all $k$-bounded partitions simply by $\mathcal{P}^k$.
\end{definition}

\begin{lemma}
\label{lemma: B partitions to odd bounded partitions bijection}
There is a size-preserving bijection $\Bpar \leftrightarrow \Bodd{2n-1}$.
\end{lemma}
\begin{proof}
Given a partition in $\Bpar$, the bijection is given by dividing all even parts repeatedly until there are no even parts left.  
\end{proof}

Let $\numpartsgtrn(\lambda)$ denote the number of parts of $\lambda$ larger than or equal to $n$.  Let $q_\lambda' := 2^{\numpartsgtrn(\lambda)}q_\lambda = 2^{\numpartsgtrn(\lambda)}Q_{\lambda_1}Q_{\lambda_2}\cdots$.

\begin{proposition}
\label{prop: type B homology spanning set}
A basis of $\homB$ over $\mathbb{Z}$ is given by $(q_\lambda')_{\lambda \in \Bpar}$.
\end{proposition}
\begin{proof}
Using the relations \eqref{Q relations}, it's clear that the $q_\lambda'$, $\lambda \in \Bpar$, span $\homB$.  By \cite[III.8.6]{mac95}, $(q_\lambda)$ for odd partitions $\lambda$ form a $\mathbb{Q}$-basis of $\Gamma_*$.  Let $\homB[k]$ be the graded part in $\homB$ of degree $k$.  Since $(q_\lambda)_{\lambda \in \Bodd{2n-1}}$ are in $\homB$, the dimension of $\homB[k]$ is at least $|\{ \lambda \in \Bodd{2n-1} \mid \lambda \vdash k\}|$.  By Lemma \ref{lemma: B partitions to odd bounded partitions bijection}, this is the same as the number of $\lambda \in \Bpar$ such that $\lambda \vdash k$; therefore, the $(q_\lambda')_{\lambda \in \Bpar}$ must be linearly independent.
\end{proof}

Let $(R_\lambda)_{\lambda \in \Bpar}$ be the basis of $\cohomB$ dual to $(q_\lambda')$ under $[\cdot, \cdot]$.  By \cite{mac95}, we have that $[q_\lambda, P_\mu] = 0$ if $\lambda > \mu$, so $R_\lambda$ is triangularly related to $(P_\lambda)_{\lambda \in \Bpar}$, which is triangularly related to the monomial symmetric functions, $m_\lambda$.  Since $q_\lambda$ and $m_\lambda$ are dual under $[\cdot, \cdot]$, we have that the coefficient of $m_\lambda$ in $R_\mu$ for $\lambda \in \Bpar$ is given by $2^{-\numpartsgtrn(\lambda)} \delta_{\lambda \mu}$.  By the duality of $q_\lambda$ and $m_\lambda$, it is clear that $\cohomB \subset \Gamma^* / \langle m_\lambda \mid \lambda_1 \geq 2n \rangle$.

\begin{proposition}
\label{type B Cauchy kernel}
Let $(w_\lambda)$ and $(t_\lambda)$ be two bases of $\homB$ and $\cohomB$, respectively, indexed by $\Bpar$.   Let $I^{(k)}$ be the ideal generated by $y_i^{k+1}$ for all $i$.  Then 
\begin{equation}
\Omega_{-1} \mod I^{(2n-1)} = \sum_{\lambda \in \Bpar} w_\lambda[X] t_\lambda[Y]
\end{equation}
if and only if $[w_{\lambda}, t_\mu] = \delta_{\lambda \mu}$ for all $\lambda, \mu \in \Bpar$.
\end{proposition}
\begin{proof}
We can write $w_\lambda = \sum_{\rho \in \Bpar} a_{\lambda \rho} q_\rho'$ and $t_\mu = \sum_{\sigma \in \Bpar} b_{\mu \sigma} R_\sigma$, where $a_{\lambda \rho}, b_{\mu \sigma} \in \mathbb{Q}$.  We therefore have
$$
[w_\lambda, t_\mu] = \delta_{\lambda \mu} \iff [  \sum_{\rho \in \Bpar} a_{\lambda \rho} q_\rho',  \sum_{\sigma \in \Bpar} b_{\mu \sigma} R_\sigma] = \delta_{\lambda \mu} \iff \sum_{\psi \in \Bpar} a_{\lambda \psi} b_{\mu \psi} = \delta_{\lambda \mu}
$$

In other words, $(a_{\lambda \rho})(b_{\mu \sigma})^T = I$, so $(b_{\mu \sigma})^T = (a_{\lambda \rho})^{-1}$.\\

On the other hand, suppose
\begin{align*}
\Omega_{-1} \mod I^{(2n-1)} &= \sum_{\lambda \in \Bpar} w_\lambda(x) t_\lambda(y) \\
\sum_{\psi \in \mathcal{P}} q_\psi m_\psi \mod I^{(2n-1)} &= \sum_{\lambda \in \Bpar}  (\sum_{\rho \in \Bpar} a_{\lambda \rho} q_\rho'(x))( \sum_{\sigma \in \Bpar} b_{\mu \sigma} R_\sigma(y)) \\
\sum_{\psi \in \mathcal{P}^{2n-1}} q_\psi m_\psi &= \sum_{\rho, \sigma \in \Bpar} (\sum_{\lambda \in \Bpar} a_{\lambda \rho}b_{\lambda \sigma}) q_{\rho}'(x)R_{\sigma}(y)\\
\end{align*}
By looking at the coefficient of $q_\lambda m_\mu$ for $\lambda, \mu \in \Bpar$ on each side, we see that we must have $\sum_{\lambda \in \Bpar} a_{\lambda \rho}b_{\lambda \sigma} = \delta_{\rho \sigma}$, i.e., $(a_{\lambda \rho})^T = (b_{\mu \sigma})^{-1}$.  The proposition follows.

\end{proof}

We now set 
\begin{align}
\repkerB  :=& \, \Omega_{-1} \mod I^{(2n-1)}\\
= & \sum_{\lambda_1 \leq 2n-1} 2^{\numpartsgtrn(\lambda)} Q_{\lambda_1}[X] Q_{\lambda_2}[X] \cdots 2^{-\numpartsgtrn(\lambda)} m_\lambda[Y],
\end{align}
so that $\repkerB$ is the reproducing kernel for the pairing $[ \cdot, \cdot]:\homB \times \cohomB \to \mathbb{Z}$.

\subsection{Type $D$ reproducing kernel}

The following analogs exist in type $D$ (again, see \cite{ee98,bm08}):

\begin{definition}
The set of type $D_n$ affine (colored) partitions is given by $\Dpar = \{ \lambda \mid \lambda_1 \leq 2n-2 \textrm{ and } \lambda \textrm{ has distinct parts of size smaller than } n\}$, with the additional information of a color, $b$ (blue) or $c$ (crimson), associated to each partition.
\end{definition}

\begin{lemma}
\label{lemma: type D partitions to odd and n-1 parts bijection}
If $n$ is odd, there is a size-preserving bijection between $\Dpar$ and the set of (uncolored) partitions $\mathcal{P}_{\textrm{odd},n-1}^{2n-2} := \{ \lambda \mid \lambda_1 < 2n-2, \lambda_i \textrm{ is odd or } \lambda_i = n-1 \textrm{ for all } i\}$.
\end{lemma}
\begin{proof}
This is similar to the type $B$ bijection, except we only act on one color: say, if a partition is colored $b$, split up all even parts as in the type $B$ case and if colored $c$, then leave parts of size $n-1$ alone.
\end{proof}

Note that if we split up all partitions regardless of color, then we get a 2-to-1 map from $\Dpar$ onto $\mathcal{P}_{\textrm{odd}}^{2n-2}$.  Let $\Dparonecolor$ be the type $D$ partitions of color $b$.  The following proposition is very similar the same as the type $B$ case.

\begin{proposition}
\label{prop: type D homology spanning set}
A basis of $\homD$ over $\mathbb{Z}$ is given by $(q_\lambda')_{\lambda \in \Dparonecolor}$.
\end{proposition}
\begin{proof}
Similar to Proposition \ref{prop: type B homology spanning set}.
\end{proof}

As in the type $B$ case, we have $\cohomD \subset \Gamma^* / \langle m_\lambda \mid \lambda_1 \geq 2n-1 \rangle$.  We also have that a reproducing kernel for the pairing $[\cdot, \cdot ]: \homD \times \cohomD \to \mathbb{Z}$ is given by
\begin{equation}
\repkerD = \sum_{\lambda_1 \leq 2n-2} 2^{\numpartsgtrn(\lambda)} Q_{\lambda_1}[X] Q_{\lambda_2}[X] \cdots 2^{-\numpartsgtrn(\lambda)} m_\lambda[Y].
\end{equation}

\subsection{Segments}
Before describing type-specific generators of $\mathbb{B}$, we borrow definitions and notation from Billey and Mitchell \cite{bm08}.

\begin{definition}
Define the sets of affine Weyl group generators:
 $$
 S = \{ s_1, s_2, \ldots, s_n\}, \quad S' = \{s_0, s_2, \ldots, s_n\}, \quad J = \{ s_2, s_3, \ldots, s_n \}.
 $$  
 
 For each $j \geq 0$, the length $j$ elements of $((\Waf)_{S'})^J$ are known as \emph{$0$-segments}, and denoted $\Sigma_0^a(j)$. In most cases, there is only one segment of a given length, in which case $a$ is omitted -- otherwise, it is used to label which segment of length $j$ is being used.  Similarly, the elements of $((\Waf)_{S})^J$ are called \emph{$1$-segments}, and are denoted $\Sigma_1^a(j)$.  Any $0$-segment or $1$-segment is known simply as a \emph{segment}.  (In \cite{bm08}, these are segments corresponding to Type II Coxeter groups).
\end{definition}

Billey and Mitchell give explicit descriptions of the segments in types $B$ and $D$, from which it follows easily that any segment is in fact in $\Z$.  In type $B$, they describe:
\begin{equation}
\Sigma_1(j) =
\begin{cases}
s_j \cdots s_3 s_2 s_1 & 1 \leq j \leq n\\
s_{2n-j} \cdots s_{n-1} s_n s_{n-1} \cdots s_3 s_2 s_1 & n < j \leq 2n-1
\end{cases}
\end{equation}
and
\begin{equation}
\Sigma_0(j) = 
\begin{cases}
s_0 & j=1\\
s_j \cdots s_3 s_2 s_0 & 1 < j \leq n\\
s_{2n-j} \cdots s_{n-1} s_n s_{n-1} \cdots s_3 s_2 s_0 & n < j \leq 2n-2\\
s_0 s_2 s_3 \cdots s_{n-1} s_n s_{n-1} \cdots s_3 s_2 s_0 & j = 2n-1.
\end{cases}
\end{equation}
In type $D$, $1$-segments are given by:
\begin{equation}
\Sigma_1^z(j) = 
\begin{cases}
s_j \cdots s_3 s_2 s_1 & 1 \leq j \leq n-2\\
s_{n-1} s_{n-2} \cdots s_3 s_2 s_1 & j = n-1 \textrm{ and } z = b\\
s_n s_{n-2} \cdots s_3 s_2 s_1 & j = n-1 \textrm{ and } z = c\\
s_{2n-j-1} \cdots s_{n-2} s_n s_{n-1} s_{n-2} \cdots s_3 s_2 s_1 & n \leq j \leq 2n-2
\end{cases}
\end{equation}
and $0$-segments are obtained from $1$-segments by interchanging $s_0$ and $s_1$ and $s_{n-1}$ and $s_n$.  There are two colors of length $n-1$ segments, and only one color for all other lengths.

It also follows from their description that if $w \in \Z \cap \Wgrass$, then $w$ is a segment (and if $w$ is $1$-Grassmannian, then $w$ is also a segment).  

\begin{theorem}[\cite{bm08}, Lemmas 3, 5]
Given $w \in \Wgrass$, where $\Waf$ is of type $B$ or $D$, $w$ has a length-decreasing factorization $r(w)$ into segments, that is, a factorization
$$
r(w) = \cdots \Sigma_0^{c_3}(\lambda_3) \Sigma_1^{c_2}(\lambda_2) \Sigma_0^{c_1}(\lambda_1)
$$
such that $\lambda_1 \geq \lambda_2 \geq \lambda_3 \geq \cdots$.
\end{theorem}
Furthermore, this factorization of $w$ is unique given that any initial product of segments is a factorization $r(u)$ of some element $u \in \Wgrass$.

\subsection{Type $B$ Pieri elements.}

Recalling the elements $\rho_i$ defined in \eqref{eq: type B rho}, we note that $\rho_i$ is the unique Grassmannian element in $\ZB$ of length $i$ for each $i$, and define $\specialgenB{r} = \specialgenB{\rho_r}$.  Let $\ZB_i$ be the length $i$ elements of $\ZB$. 

\begin{proposition}
\label{prop: type B special generators formula}
For $1 \leq r \leq 2n-1$, 
\begin{equation}
\label{type B special generators equation}
\specialgenB{r} = \sum_{w \in \ZB_r} 2^{\stat(w)-\chi(r < n)}A_w
\end{equation}
\end{proposition}
\begin{proof}
See \textsection \ref{subsection: proof of special generators formula}.
\end{proof}

We call the elements defined in Proposition \ref{prop: type B special generators formula} \emph{Pieri elements}.  By analogy to the type $A$ case, they are noncommutative $k$-Schur functions corresponding to the special elements $\rho_i$, and the relations among Pieri elements are given in Proposition \ref{prop: type B special generators relation}.  

\begin{proposition}
\label{prop: type B special generators relation}
The Pieri elements $\specialgenB{i} \in \mathbb{B}$ satisfy 
\begin{equation}
\label{eq: type B special generators equation}
\sum_{r + s = 2m} (-1)^r 2^{-\chi(r \geq n) - \chi(s \geq n)} \specialgenB{r}\specialgenB{s} = 0
\end{equation}
\end{proposition}
\begin{proof}
Suppose $w \in \Bgrass$, $\ell(w) \leq 2n-1$.  We analyze the coefficient of $\specialgenB{w}$ in $\specialgenB{2m-i}\specialgenB{i}$, for any $i$.  To that end, suppose that $w = uv$ with $u \in \Z_{2m-i}$ and $v \in \Z_i$.  Since any right factor of a Grassmannian element is Grassmannian, we must have that $v = \cdots s_3 s_2 s_0$ is a segment of length $i$.  By \cite{bm08}, $w$ has a length-decreasing factorization into segments; if $w = uv$ with $u,v \in \ZB$, then the factorization of $w$ into segments can involve at most two segments.

On a case by case basis, depending on the relation of $m, i,$ and $k_0$, we can verify that 
$\specialgenB{w}$ appears in the product $\specialgenB{2m-i}\specialgenB{i}$ if and only if $k_1 \leq i \leq k_0$.  Furthermore, the above obvious factorization of $w$ into $u,v$ with $u \in \Z_{2m-i}$, $v \in \Z_{i}$ is the only such factorization, since there is only one reduced word for $\Sigma_0(k_0)$.  

Similarly, a case analysis (depending on $k_0 \geq n$ or $k_0 < n$) of the coefficients that appear in the left hand side of \eqref{eq: type B special generators equation} and noting that 
\[
\stat(\Sigma_0(i)) -1 = \left \{ \begin{array}{cc} -1 & \textrm{ if } i \geq n\\ 0 &\textrm{ else } \\ \end{array} \right.,
\]
finishes the proof.
\end{proof}

\begin{theorem}
\label{theorem: type B coproduct formula}
In type $B$, for $1 \leq r \leq 2n-1$,
\begin{equation}
\label{type B coproduct equation}
\phi_0^{(2)}(\Delta(\specialgenB{r})) = 1 \otimes \specialgenB{r} + \specialgenB{r} \otimes 1 + \sum_{1 \leq s < r}2^{\chi(r \geq n>r-s \textrm{ and } n > s)} \specialgenB{s} \otimes \specialgenB{r-s} 
\end{equation}
\end{theorem}

\begin{proof}
See \textsection \ref{subsection: proof of coproduct formula}.
\end{proof}

\subsection{Type $B$ affine Stanley symmetric functions.} Define $\Omega_{-1}^{\mathbb{B}} \in \mathbb{B} \hat{\otimes} \cohomB$ by taking the image of $\repkerB$ under $j_0 \circ \PhiB : \homB \to \mathbb{B}$.

\begin{align}
\label{eq: type B omega as sum over partitions}
\Omega_{-1}^{\mathbb{B}} & = \sum_{\lambda_1 \leq 2n-1} \specialgenB{\lambda_1}\specialgenB{\lambda_2}\cdots \otimes 2^{-\numpartsgtrn(\lambda)} m_\lambda[Y] \\
&= \sum_{\substack{\alpha \\ \alpha_i \leq 2n-1}} \specialgenB{\alpha_1}\specialgenB{\alpha_2}\cdots \otimes 2^{-\numpartsgtrn(\lambda)} y^\alpha,
\end{align}
where the second equality follows because $\mathbb{B}$ is a commutative algebra.

Define $\ASSFB{w}[Y]$ by 
\begin{equation}
\Omega_{-1}^{\mathbb{B}} = \sum_{w \in \WB} A_w \otimes \ASSFB{w}[Y]
\end{equation}

Recalling that $\specialgenB{r} = \sum_{w \in \ZB_r} 2^{\stat(w) - \chi(r < n)} A_w$, it is not hard to see that this definition matches that of Definition \ref{def: affine Stanley symmetric functions}.  We also note that
\begin{equation}
\label{eq: repkerBalg}
\repkerBalg = \sum_{w \in \Bgrass} \specialgenB{w} \otimes \ASSFB{w}[Y]
\end{equation}
by Theorem \ref{theorem: j_0 is isomorphism}.

\subsection{Type $D$ Pieri elements.}
\label{subsection: type D special generators} Recall that the elements $\rho_i \in \WD$ defined in \eqref{eq: type D rho}, as well as $\rho_{n-1}^{(1)}$ and $\rho_{n-1}^{(2)}$.  Let $\specialgenD{r} = \mathbb{P}_{\rho_r}$ for $1 \leq r \leq 2n-2, r \neq n-1$, and let $\specialgenD{n-1} = \frac{1}{2}(\mathbb{P}_{\rho_{n-1}^{(1)}}^D + \mathbb{P}_{\rho_{n-1}^{(2)}}^D)$.

Given a reduced word for $v \in \ZD$, let $v_+$ be the element of $\WD$ with reduced word given by the subword of $v$ of letters with index greater than $j$, and let $v_-$ have reduced word given by the subword of $v$ of letters with index less than $j$.
\begin{definition}
\label{type D epsilon definition}
We define the special element $\epsilon \in \nilCox$ by stating that the coefficient of $A_{\rho_{n-1}^{(1)}}$ in $\epsilon$ is 1, the coefficient of $A_{\rho_{n-1}^{(2)}}$ is $-1$, and all other coefficients are given by the following symmetries.  
\begin{enumerate}
\item For any $2 \preceq j \preceq n-2$, $\pm A_{v_- j v_+} \in \epsilon \implies \pm A_{v_+ j v_-} \in \epsilon$.
\item For any $2 \preceq j \preceq n-2$, $\pm A_{v_- v_+ j} \in \epsilon \implies \pm A_{j v_- v_+} \in \epsilon$.
\item If $\pm A_w \in \epsilon$ and $w'$ is obtained from $w$ by swapping $n$ and $n-1$ or swapping $0$ and $1$ in a reduced word for $w$, then $\mp A_{w'} \in \epsilon$.
\item $\pm A_{s_n v} \in \epsilon \implies \mp A_{v s_n} \in \epsilon$.
\item $\pm A_{s_{n-1} v} \in \epsilon \implies \mp A_{v s_{n-1}} \in \epsilon$.
\end{enumerate}
Thus the coefficients in $\epsilon$ are all $\pm 1$, and if $\epsilon = \sum_{w} c_w A_w$, then $c_w \neq 0$ if and only if $\Supp(w) = \Iaf$.
\end{definition}

It is not clear from the definition that the element $\epsilon$ is well-defined.  For $w \in \WD$, given a reduced word $\text{u} = u_1 u_2 \cdots u_{\ell(w)} \in R(w)$, let $\widehat{des}(u)$ denote the number of $i$ such that $ u_i > u_{i+1}$, and let $\widehat{des}(w) = \min (\widehat{des}(u) \mid u \in R(w))$.  Then it is not hard to see that knowing $\supp(w)$ and the parity of $\widehat{des}(w)$ is enough to give the sign of $A_w$ in $\epsilon$ (e.g., $\widehat{des}(w)$ is invariant under swaps of type (1) or (2), and changes parity under swaps of type (4) or (5), while $\supp(w)$ changes under swaps of type (3)).  Therefore, $\epsilon$ is well-defined.

Note also that $\epsilon$ has only two Grassmannian terms, $A_{\rho_{n-1}^{(1)}}$ and $A_{\rho_{n-1}^{(2)}}$.  It follows from Proposition \ref{prop: type D special generators formula} that $\epsilon = \frac{1}{2}(\specialgenD{\rho_{n-1}^{(1)}} - \specialgenD{\rho_{n-1}^{(2)}})$.

\begin{proposition}
\label{prop: type D special generators formula}
For $1 \leq r \leq 2n-2$, 
\begin{equation}
\label{eq: type D special generators equation}
\specialgenD{r} = \sum_{w \in \Z_r} 2^{\stat(w)-\chi(r < n)}A_w
\end{equation}
Furthermore, the special element $\epsilon$ given by Definition \ref{type D epsilon definition} lies in $\B$.
\end{proposition}

\begin{proof}
See \textsection \ref{subsection: proof of special generators formula}.
\end{proof}

The relations among the Pieri elements are given in the following proposition.
\begin{proposition}
\label{prop: type D special generators relation}
The elements $\specialgenD{i} \in \mathbb{B}$ satisfy 
\begin{equation}
\label{eq: type D special generators equation 1}
\sum_{r + s = 2m} (-1)^r 2^{-\chi(r \geq n) - \chi(s \geq n)} \specialgenD{r}\specialgenD{s} = 0
\end{equation}
and
\begin{equation}
\label{eq: type D special generators equation 2}
(\specialgenD{n-1} + \epsilon)(\specialgenD{n-1} - \epsilon) - \specialgenD{n-2}\specialgenD{n} + \cdots \pm \specialgenD{0}\specialgenD{2n-2} = 0
\end{equation}
\end{proposition}
\begin{proof}
Equation \eqref{eq: type D special generators equation 1} follows by a similar argument to Proposition \ref{prop: type B special generators relation}, and \eqref{eq: type D special generators equation 2} follows from using the symmetry relations of $\epsilon$ to find the coefficients that appear in $\epsilon^2$.
\end{proof}
 
 \begin{theorem}
\label{theorem: type D coproduct formula}
In type $D$, for $1 \leq r < 2n-2$,
\begin{align}
\phi_0^{(2)}(\Delta(\specialgenD{r})) &= 1 \otimes \specialgenD{r} + \specialgenD{r} \otimes 1 + \sum_{1 \leq s < r}2^{\chi(r \geq n>r-s \textrm{ and } n > s)} \specialgenD{s} \otimes \specialgenD{r-s}  \\
\phi_0^{(2)}(\Delta(\specialgenD{2n-2})) &= 1 \otimes \specialgenD{2n-2} + \specialgenD{2n-2} \otimes 1 + \sum_{1 \leq s < 2n-2, s \neq n-1} \specialgenD{s} \otimes \specialgenD{2n-2-s}\\ 
\nonumber &\quad + 2\specialgenD{n-1} \otimes \specialgenD{n-1} + (-1)^{n-1}\cdot 2 \epsilon \otimes \epsilon \\
\phi_0^{(2)}(\Delta(\epsilon)) &= 1 \otimes \epsilon + \epsilon \otimes 1
\end{align}
\end{theorem}
\begin{proof}
See \textsection \ref{subsection: proof of coproduct formula}.
\end{proof}

 Based on the relations above and the scarcity of primitive elements, we conjecture that the Pieri elements $\specialgenD{i}$ and $\epsilon$ correspond with the $\sigma_i$ and $\epsilon$ of Bott \cite{bott58}, respectively.

\subsection{Type $D$ affine Stanley symmetric functions} Similar to the type $B$ case, define $\repkerDalg \in \mathbb{B}/\langle \epsilon \rangle$ as the image of $\repkerD$ under $j_0 \circ \PhiD$ (by abuse of notation, let $j_0$ also denote the isomorphism $H_*(\GrD)/\langle \xi_{\rho_{n-1}^{(1)}} - \xi_{\rho_{n-1}^{(1)}} \rangle \to \mathbb{B}/\langle \epsilon \rangle$ induced by $j_0$).  Then
\begin{equation}
\repkerDalg = \sum_{{\lambda_1 \leq 2n-2}} \specialgenD{\lambda_1} \specialgenD{\lambda_2} \cdots \otimes 2^{-p_{\geq n }(\lambda)} m_\lambda[Y].
\end{equation}

As in type $B$, we define $\ASSFD{w}[Y]$ by
\begin{align}
\repkerDalg &= \sum_{w \in \WD} A_w \otimes \ASSFD{w}[Y]\\
&= \sum_{w \in \Dgrass} \specialgenD{w} \otimes \ASSFD{w}[Y].
\end{align}

\subsection{Proofs of Propositions \ref{prop: type B special generators formula} and \ref{prop: type D special generators formula}}
\label{subsection: proof of special generators formula}

Propositions \ref{prop: type B special generators formula} and \ref{prop: type D special generators formula} are proved following the scheme laid out in \cite{lss10}.  Due to the close similarities between the special orthogonal and symplectic Pieri factors, many of the results have similar (and quite tedious) proofs; for brevity's sake, we briefly summarize only the main idea, and refer the reader to \cite{lss10, pon10} for full details.

The propositions are proved using Lemma \ref{lss10 lemma 4.7}. Because of that, we are mostly concerned with the set of Pieri factors, and so we define the set of \emph{Pieri covers} of $v \in \Z$ to be $C_v = \{w \in \Z \mid w \gtrdot v \}$.  Then Propositions \ref{prop: type B special generators formula} and \ref{prop: type D special generators formula} are proved by showing the following:

\begin{proposition}
\label{type free cover coroot sum}
Let $\Waf = \WB$ or $\WD$, and $\Z = \ZB$ or  $\ZD$, respectively.  Let $v \in \Z$ with $\ell(v) < 2n-1$ (resp. $\ell(v) < 2n-2$).  Then
\begin{equation}
\sum_{w \in C_v} 2^{\Dstat{w}} \asscovercoroot = 2^{\Dstat{v}}K
\end{equation}
where $K$ is the canonical central element.
\end{proposition}

For the type $D$ case, we also need the following lemma.

\begin{lemma}
\label{epsilon cover coroot sum}
Let $v \in \ZD$ with $\ell(v) = n-2$, and let $C_v = \{w \in \ZD \mid w \gtrdot v \}$.  If $\epsilon = \sum_w c_w A_w$, then $\sum_{w \in C_v} c_w = 0$.
\end{lemma}

In general, the proof of Proposition \ref{type free cover coroot sum} proceeds by analyzing very specifically the types of reduced words that elements of $\Z$ may have.  Given an element $v$ of $\Z$ of any type, it is possible to define a canonical reduced word for $v$.  One can then classify all possible Bruhat covers in $\Z$ of $v$, obtained by inserting ``missing'' letters into the normal reduced word for $v$.  Through lengthy, case-by-case calculations, it is possible to verify Proposition \ref{type free cover coroot sum} and Lemma \ref{epsilon cover coroot sum}.  Given the approach in \cite{lss10} and Definitions \ref{def: type B pieri factors} and \ref{def: type D pieri factors}, the necessary changes to prove the types $B$ and $D$ cases are mostly clear.  One small change is the necessary addition of an \emph{extra-special cover} when classifying the covers of a given element $v$ (see \cite{pon10}).

\subsection{Proofs of Theorems \ref{theorem: type B coproduct formula} and \ref{theorem: type D coproduct formula}}
\label{subsection: proof of coproduct formula}

The coproduct formulas for type $B$ and $D$ Pieri elements can also be proved using generalizations of the approach in \cite{lss10}.  In \cite{pon10}, we introduced an additional sign-reversing involution to clean up the proof; however, this is not strictly necessary.  Therefore, again for brevity's sake, we summarize the main results and give a general sketch of the proof.  Readers who would like details are referred to the references listed above.

In \cite{lss10}, a type-free coproduct rule for elements $A_w$ of the nilCoxeter algebra is given (\cite{lss10}, Proposition 7.1), involving operations on reduced words of $w$.  By Proposition \ref{prop: type B special generators formula} and Theorem \ref{theorem: j_0 is isomorphism}, we can reduce to considering only Grassmannian elements, which allows us to classify terms that will appear in $\Delta(A_w)$.  Again, careful case-by-case analysis shows that the proper terms appear.

\section{Proofs of Main Theorems}
\label{section: proof of main theorems}

\subsection{Type $B$} Recall that a length-decreasing factorization of $w$ is a factorization $w = v^1 v^2 \cdots v^s$ such that $\ell(v^1) \leq \ell(v^2) \leq \cdots \leq \ell(v^s)$.  A \emph{maximal} length-decreasing factorization is one such that each $v^i$ is as large as possible given $v^{i+1}, \ldots, v^s$.

\begin{lemma}
\label{type B maximal factorization}
A maximal length-decreasing factorization of $w \in \Bgrass$ into elements of $\ZB$ is a factorization of $w$ into segments.
\end{lemma}
\begin{proof}
The only elements of $\ZB$ that are not segments are elements containing both $s_0$ and $s_1$ adjacent to each other.  Suppose $w = v^1 v^2 \cdots v^s$ is a length-decreasing factorization of $w$ into elements of $\ZB$.  By induction, suppose $v^{i+1}$ is a segment, beginning with $0$ (resp. $1$).  Then if $v^i = \cdots s_k$, we must have $k = 1$ (resp. $k=0$).  If not, then $w$ is either not Grassmannian or $v^{i+1}$ is not of maximal size.  By the same reasoning, $1$ (resp. $0$) must be the only descent of $v^i$, so that $v^i$ is $1$-Grassmannian (resp. $0$-Grassmannian).  There is only one $1$-Grassmannian (resp. $0$-Grassmannian) element in $\ZB$ of each given length, and these are exactly the segments.
\end{proof}

Billey and Mitchell give a bijection between $w \in \Bgrass$ and the affine type $B$ partitions, $\Bpar$.  Their bijection is given by taking a Grassmannian $w$ and factoring it into segments; the lengths of the segments give the associated partition.  By Lemma \ref{type B maximal factorization}, this is the same as factoring a Grassmannian $w$ into elements of $\ZB$.

\begin{proposition}
\label{type B ASSF are linearly independent}
The functions $\{ \ASSFB{w} \mid w \in \Bgrass \}$ are linearly independent.
\end{proposition}
\begin{proof}
By Lemma \ref{type B maximal factorization}, $\ASSFB{w} = \sum_{\mu \leq \lambda(w)} a_{\mu,\lambda(w)} m_\mu$, where $\lambda(w)$ is the largest partition associated to a factorization of $w$ into Pieri factors, and $\leq$ is lexicographic ordering on partitions.  Further, if $w \neq v$, then $\lambda(w) \neq \lambda(v)$, so $A = (a_{\mu, \lambda})$ is triangular.
\end{proof}

There is a surjective ring homomorphism $\theta: \Lambda \to \Gamma_*$ defined by $\theta(h_i) = Q_i$ (see \cite[Ex. III.8.10]{mac95}).  Let $\iota : \Gamma^* \to \Lambda$ be the inclusion map.

\begin{lemma}[\cite{lss10},Lemma 2.1]
\label{hall inner product versus hall-littlewood inner product}
Given $f \in \Gamma^*$, $g \in \Lambda$, $\langle \iota(f), g \rangle = [ f, \theta(g)].$
\end{lemma}

\begin{proposition}
\label{type B ASSF span dual space}
$\cohomB$ is spanned by $\{ \ASSFB{w} \mid w \in \WB^0 \}$.
\end{proposition}
\begin{proof}

Given $\lambda \in \Bpar$, consider $g' = 2^{\numpartsgtrn(\lambda)}h_\lambda$, so that $\theta(g') = q_\lambda' $.  Let $w \in \Bgrass$ such that the maximal factorization of $w$ into segments corresponds to $\lambda$ under the bijection of \cite{bm08}.  Then it is easy to see that the coefficient of $m_{\lambda}$ in $\ASSFB{w}$ is $2^{-\numpartsgtrn(\lambda)}$.  By Lemma \ref{hall inner product versus hall-littlewood inner product}, $[\ASSFB{w},g] = \langle \ASSFB{w}, g' \rangle = 1$.  Furthermore, given $\psi \geq \lambda$ (in lexicographic order), with $\psi \in \Bpar$ and $g'' = 2^{\numpartsgtrn(\psi)} Q_{\psi_1}Q_{\psi_2}\cdots$, we have $[\ASSFB{w},g''] = 0$.  Therefore, $\cohomB = \textrm{Hom}_\mathbb{Z}(\homB,\mathbb{Z})$ is spanned by $\{ \ASSFB{w} \mid w \in \WB^0 \}$.

\end{proof}

\begin{proof}[Proof of Theorem \ref{theorem: type B cohomology isomorphism}]
Given Proposition \ref{prop: type B special generators relation}, Theorem \ref{theorem: type B coproduct formula}, and Theorem \ref{theorem: j_0 is isomorphism}, we may borrow the proof of Theorem 1.3 of \cite{lss10}, and proceed analogously.  We have that $\PhiB:\homB \to H_*(\textrm{Gr}_{SO_{2n+1}(\mathbb{C})})$ is a bialgebra morphism.  Since both are graded commutative and cocommutative Hopf algebras, it must be a Hopf algebra morphism.  We defined $\PsiB:H^*(\textrm{Gr}_{SO_{2n+1}(\mathbb{C})}) \to \cohomB$ by $\xi^w \to \ASSFB{w}$ for $w \in \Bgrass$.  We show that $\PhiB$ and $\PsiB$ are dual with respect to the pairing $\langle \cdot, \cdot \rangle : H_*(\textrm{Gr}_{SO_{2n+1}(\mathbb{C})}) \times H^*(\textrm{Gr}_{SO_{2n+1}(\mathbb{C})}) \to \mathbb{Z}$ induced by the cap product the pairing $[\cdot,\cdot]: \homB \times \cohomB \to \mathbb{Z}$.  We want to show that $\langle \PhiB(f), \xi^w \rangle = [f, \PsiB(\xi^w)]$, for all $f$ in a spanning set of $\homB$.  We have:
\begin{align}
[2^{\numpartsgtrn(\lambda)}Q_{\lambda_1}\cdots Q_{\lambda_l},\PsiB(\xi^w)] & =  [2^{\numpartsgtrn(\lambda)}Q_{\lambda_1}\cdots Q_{\lambda_l}, \ASSFB{w}]\\
&=[2^{\numpartsgtrn(\lambda)}Q_{\lambda_1}\cdots Q_{\lambda_l}, \langle \Omega_{-1}^\mathbb{B}, \xi^w \rangle ]\\
&= \langle [ 2^{\numpartsgtrn(\lambda)}Q_{\lambda_1}\cdots Q_{\lambda_l}, \Omega_{-1}^\mathbb{B}], \xi^w \rangle \\
&= \langle \specialgenB{\lambda_1} \cdots \specialgenB{\lambda_\ell}, \xi^w \rangle \\
&= \langle \PhiB(2^{\numpartsgtrn(\lambda)}Q_{\lambda_1}\cdots Q_{\lambda_l}), \xi^w \rangle.
\end{align}
The second equality holds by identifying $\mathbb{B}$ and $H_*(\Gr)$ and \eqref{eq: repkerBalg}.

 The fourth equality holds by \eqref{eq: type B omega as sum over partitions} and (2.24) of \cite{lss10}, which states that 
 
 \noindent $[2^{\numpartsgtrn(\lambda)} Q_{\lambda_1} \cdots Q_{\lambda_\ell}, f ] = 2^{\numpartsgtrn(\lambda)}$ times the coefficient of $m_\lambda$ in $f$.  The other equalities hold by definition.

 Therefore, $\PsiB$ is a Hopf algebra morphism.  By Lemmas \ref{type B ASSF are linearly independent} and \ref{type B ASSF span dual space}, $\PsiB$ is a bijection; therefore, it is an isomorphism.
 
 \end{proof}

\subsection{Positivity of type $B$ $k$-Schur functions}
\label{subsection: positivity of type B k-Schurs}

In \cite{lam09}, Lam offers a point of view relating geometric positivity to Schur-positivity (or Schur $P$-positivity) of symmetric functions, based on the following theorem.  Suppose we have an embedding of affine Grassmannians, $\iota: H_*(\Gr) \to H_*(\textrm{Gr}_{G'})$.  Then:

\begin{theorem}[\cite{lam09}]
\label{theorem: lam09, theorem 1}
For any $v \in \Wgrass$, the pushforward $\iota_*(\xi_{v}) \in H_*(\Gr)$ of a Schubert class is a nonnegative linear combination of Schubert classes $\{ \xi_w \mid w \in (\widetilde{W}')^0 \}$ of $H_*(\textrm{Gr}_{G'})$.
\end{theorem}

Letting $G' = SL(\infty, \mathbb{C})$, we have $H_*(\textrm{Gr}_{G'}) \cong \Lambda$, and the Schubert basis is identified with Schur functions.  

For $G = SL(n,\mathbb{C})$ or $G = Sp(2n, \mathbb{C})$, the maps $H_*(\Gr) \to H_*(\textrm{Gr}_{SL(m, \mathbb{C})}) \to H_*(\textrm{Gr}_{SL(\infty,\mathbb{C})})$ are inclusions, and it can be shown that every Schubert class is a Schur-positive symmetric function (similarly, the inclusions $H_*(\textrm{Gr}_{Sp(2n, \mathbb{C})}) \to H_*(\textrm{Gr}_{Sp(2m, \mathbb{C})})$ give Schur $P$-positivity of type $C$ homology Schubert polynomials).  Lam's arguments can be easily adapted to show that for $G = \Spin(2n+1,\mathbb{C})$, homology Schubert classes can be identified with Schur-positive and Schur $P$-positive symmetric functions.  By considering the natural inclusions $SO(n) \hookrightarrow SU(n)$ and $SO(2n+1) \hookrightarrow SO(2n+3)$ and proofs analogous to those in \cite{lam09}, we have the following.

\begin{proposition}
\label{prop: type B topological inclusions}
The induced maps on homology 
\begin{itemize}
\item[(1)] $H_*(\Omega_0 SO(2n+1)) \to H_*(\Omega_0 SO(2n+3))$ and
\item[(2)] $H_*(\Omega_0 SO(2n+1)) \to H_*(\Omega SU(2n+1))$
\end{itemize}
are Hopf-inclusions.  Furthermore, (1) is a $\mathbb{Z}$-module isomorphism in degrees less than $4n-1$.
\end{proposition}

\begin{proposition}
\label{prop: type B k-schur are schur P positive}
The symmetric functions $\kSB{w}$ expand positively in terms of the $\{ \tilde{G}_v^{B_{n+1}}\}$ basis.
\end{proposition}

\begin{proposition}
Given $w \in \Bgrass$, if $n > \ell(w)$, then $\ASSFB{w}$ is a Schur $P$-function.  By duality, the type $B$ $k$-Schur functions $\kSB{w}$ are Schur $Q$-functions.
\end{proposition}
\begin{proof}
If $n > \ell(w)$, then $n \notin \supp(w)$.  Given a component $c$ of $w$ (that is, a subword with connected support), $c$ can be written uniquely as a subword of $s_{n}s_{n-1} \cdots s_2 s_1 s_0 s_2 \cdots s_{n-1} s_n$; we say such a subword is a ``$V$.''  Furthermore, all distinct components commute.  Therefore, write $w = c_1 c_2 \cdots c_k$, where $c_i$ is the $i$th component of $w$ (suppose they are in order with respect to $\succ$, so that $c_1$ is the component with the largest indices).  Then each $c_i$ has a reduced word that is a $V$ given by $c_i^1 m_i c_i^2$, where $m_i$ is the minimum element in $c_i$.  One way to write $w$ as a $V$ is 
$$
w = c_1^1m_1 c_2^1 m_2 \cdots c_{k-1}^1 m_{k-1} c_k c_{k-1}^2 \cdots c_2^2 c_1^2.
$$
However, for each $1 \leq i \leq k-1$, we can move $m_i$ from immediately prior to $c_{i+1}^1$ to immediately after $c_{i+1}^2$ and still have a $V$.  
Therefore, there are $2^{c-1}$ possible ways to write $w$ as a $V$, where $c$ is the number of components of $w$.  Note that $c = cc(w)$ unless neither $0$ nor $1$ is in the support of $w$, in which case $c = cc(w) - 1$. 

In \cite{tklam95}, T.K. Lam defines type $D$ Stanley symmetric functions (see also \cite{bh94}) $H_w(x)$ by
\begin{equation}
D(x_1)D(x_2) \cdots = \sum_{w} H_w(x) w,
\end{equation}
where $D(x) = (1+xu_{n-1})\cdots (1+xu_2)(1+xu_1)(1+xu_0)(1+xu_2) \cdots (1+xu_{n-1})$, and the $u_i$ are generators for the finite nilCoxeter algebra of type $D$ such that $0$ and $1$ commute, with the action of $u_i$ on permutations as defined in the introduction.  Recalling our definition of type $B$ affine Stanley symmetric functions (Definition \ref{def: affine Stanley symmetric functions}), we see that $\ASSFB{w}(x) = H_w(x)$, since the coefficient of a given $v^i$ in $D(x)$ is $cc(v^i)$.  By \cite[Theorem 4.35]{tklam95}, $H_w(x)$ expands as a nonnegative sum of Schur $P$-functions.
\end{proof}

\begin{corollary}
The type $B$ $k$-Schur functions $\kSB{w}$ expand positively in terms of Schur $Q$-functions.
\end{corollary}

\begin{proposition}
\label{prop: type B k-schur are k-schur positive}
The type $B$ $k$-Schur functions $\kSB{w}$ are $k$-Schur positive, with $k = 2n$.
\end{proposition}

\begin{proof}[Proof of Theorem \ref{theorem: ASSF form positive basis}] The proof of Theorem \ref{theorem: ASSF form positive basis} may also follow its analog in \cite{lss10}, after establishing the above results.  The affine Stanley symmetric functions are symmetric by their alternative definition using the type $B$ reproducing kernel and the commutativity of $\mathbb{B}$.  They form a basis for $w \in \Bgrass$ because the Schubert classes $\xi^w$ for $w \in \Bgrass$ form a basis of $H^*(\GrB)$.  Graham \cite{gra01} and Kumar \cite{kum02} showed the positivity of the structure constants.  By duality, the coproduct structure constants of $\{ \ASSFB{w} \mid w \in \Bgrass \}$ equal the product structure constants of $\{\xi_w \mid w \in \Bgrass \}$, which are nonnegative by work of Peterson \cite{pet97} and Lam and Shimozono \cite{ls10} (see \cite{lss10}).  
\end{proof}

\begin{proof}[Proof of Theorem \ref{theorem: type B pieri rule}]
Follows from Theorem \ref{theorem: homology pieri rule} and Proposition \ref{prop: type B special generators formula}.
\end{proof}

\subsection{Main Theorems -- Type $D$}
The type $D$ results are not as satisfying or complete, given the constraints discussed in the introduction.

\begin{proposition}
\label{prop: type D ASSF span dual space}
$\cohomD$ is spanned by $\{ \ASSFD{w} \mid w \in \WD^0 \}$.  The functions $\{\ASSFD{w} \mid w \in \WD^0 \}$ are linearly independent, except for $\ASSFD{w} = \ASSFD{w'}$ if there are reduced words for $w$ and $w'$ that differ only by swapping some occurrences of $n$ and $n-1$.
\end{proposition}
\begin{proof}
The proof of the first statement is similar to the proof of Proposition \ref{type B ASSF span dual space}.  The second statement also follows a similar proof to that of \ref{type B ASSF are linearly independent}, except if $\rho_{n-1}^{(i)}$ occurs as a segment in the canonical decomposition of $w$ into segments, in which case swapping $\rho_{n-1}^{(1)}$ and $\rho_{n-1}^{(2)}$ in $w$ will give another element $w'$ such that $\ASSFD{w} = \ASSFD{w'}$.
\end{proof}

Given the setup and results we have for type $D$, a proof of Conjecture \ref{conj: type D conjecture} will likely follow the same scheme as the type $B$ proofs.

\begin{proof}[Proof of Theorem \ref{theorem: type D pieri rule}]
Follows from Theorem \ref{theorem: homology pieri rule} and Proposition \ref{prop: type D special generators formula}.
\end{proof}

\section{Type-free Pieri factors} Given the type-specific definitions of Pieri factors, it is easy to prove our type-free description for the classical types.
\begin{proposition}
\label{prop: type-free pieri factors}
The Pieri factors given in \ref{subsection: type free results} match with those of with the corresponding set of affine Weyl group elements given in type $A$ (\cite[Definition 6.2]{lam08}), type $C$ (\cite[\textsection 1.5]{lss10}), and types $B$ and $D$ (Definitions \ref{def: type B pieri factors} and \ref{def: type D pieri factors}).
\end{proposition}
\begin{proof}
We use computations done by Pittman-Polletta \cite{pp10}, although others may have done similar computations (\cite{fl07}).  Pittman-Polletta computes the reduced word in the affine Weyl group corresponding to the image of the translation by $\nu(\fincoweight{1})$ for every type (in his notation, this is the Weyl group element $W_1$, and his $w$ correspond to our $w^{-1}$).  We can use that description to find a description of all length-maximal Pieri factors.  
We proceed using the descriptions of Pieri factors given in \cite{lam08} and \cite{lss10} for the type $A$ and $C$ cases, respectively.

\begin{itemize}
\item[\textbf{Type $A$:}] In this case, Pieri factors are the Bruhat order ideal generated by length-maximal cyclically decreasing words \cite{lam08}.  The reduced word in the affine Weyl group corresponding to translation by $\nu(\fincoweight{1})$ is $s_0s_{n-1}s_{n-2}\cdots s_3s_2$.  It is not hard to compute $\tau$ such that $t_{\nu(\fincoweight{1})} = \tau s_0s_{n-1}\cdots s_2$ by looking at the action of $s_0 \cdots s_2$ on simple roots -- it is the Dynkin diagram automorphism that sends $i$ to $i+1$.  Then for $w \in W$, we have 
\begin{align*}w t_{\nu(\fincoweight{1})} w^{-1} &= w \tau s_0 \cdots s_2 w^{-1}\\& = \tau (\tau^{-1} w \tau) s_0 \cdots s_2 w^{-1} \\&= \tau s_{w_1 - 1} \cdots s_{w_\ell - 1} s_0 \cdots s_2 s_{w_\ell} \cdots s_{w_1},
\end{align*} where $w_1 \cdots w_\ell$ is a reduced word for $w$.  One can check that given any cyclically decreasing word $s_r s_{r-1} \cdots s_1 s_0 s_{n-1} \cdots s_{r+2}$, the result of multiplying $s_{i-1} s_r \cdots s_{r+2} s_i$ is another cyclically decreasing word.  For example, if $i \neq r+1$, then we have
\begin{align}
s_{i-1}s_r s_{r-1} \cdots s_1 s_0 s_{n-1} \cdots s_{r+2}s_i & = s_r \cdots s_{i+1} s_{i-1} s_i s_{i-1} s_i s_{i-2} \cdots s_{r+2}\\
&= s_r \cdots s_{i+1} s_i s_{i-1} \cdots s_{r+2}
\end{align}
so the cyclically decreasing word is unchanged.  If $i = r+1$, then multiplication by $s_{i-1}$ on the left and $s_i$ on the right will rotate the reduced word.  It is clear that we can get any maximal-length cyclically decreasing word in this manner; therefore, the elements $\{t_{w(\nu(\fincoweight{1}))}, w \in S_n\}$ correspond to the type $A$ Pieri factor generators given in \cite{lam08}.

\item[\textbf{Type $B$:}] In this case, Pieri factors are given in Definition \ref{def: type B pieri factors}.  For type $B$, the fundamental coweight $\fincoweight{1}$ does not lie in the coroot lattice; therefore, translation by $\nu(\fincoweight{1})$ must involve a nontrivial Dynkin diagram automorphism.  There is only one choice for such an automorphism -- the map that exchanges $0$ and $1$.  Therefore, by \cite{pp10}, $t_{\nu(\fincoweight{1})} = \tau s_0 s_2 \cdots s_n \cdots s_2 s_0$, where $\tau$ exchanges $0$ and $1$, and we have $t_{w\nu(\fincoweight{1})} = \tau w' s_0 s_2 \cdots s_n \cdots s_2 s_0 w^{-1}$, where $w'$ is obtained from $w$ by switching any occurrences of $0$ into $1$, and vice versa.  These match with the description of maximal-length Pieri factors of type $B$, via similar calculations to the type $A$ case.

\item[\textbf{Type $C$:}] In this case, maximal-length Pieri factors are given by conjugates of the affine Weyl group element with reduced word $s_1s_2 \cdots s_{n-1} s_n s_{n-1}\cdots s_2s_1s_0$ \cite{lss10}.  For type $C$, $\nu(\fincoweight{1})$ is in the span of the coroots; 
therefore, translation by $\nu(\fincoweight{1})$ lies in the affine Weyl group, and acting on $\nu(\fincoweight{1})$ by finite Weyl group elements corresponds to conjugation.

\item[\textbf{Type $D$:}] In this case, Pieri factors are given in Definition \ref{def: type D pieri factors}.  The translation $t_{\nu(\fincoweight{1})}$ corresponds to the affine Weyl group element $v$ where
$$v = s_0 s_2 \cdots s_{n-2} s_{n-1} s_n s_{n-2} \cdots s_2 s_0.$$  As in type $A$, a simple calculation shows that $t_{\nu(\fincoweight{1})} = \tau v$, where $\tau$ exchanges $0$ and $1$ and $n$ and $n-1$.  As in the cases above, this is easily seen to correspond to the description of type $D$ Pieri factors in terms of reduced words.
\end{itemize}

\end{proof}


\section{Appendix}

The following are examples of affine Stanley symmetric functions and their duals for $w \in \widetilde{B}_3^0$ (they have been implemented in the Sage open-source mathematical software package).  Affine Stanley symmetric functions are expanded in terms of monomial symmetric functions indexed by $\lambda$ with $\lambda_1 \leq 5$, since we are working in the quotient ring.  Type $B$ $k$-Schur functions are expanded in terms of Schur $Q$-functions.

\scriptsize
$$
\begin{array}{|c|c|c|}
\hline
w & \ASSFB{w} & \kSB{w}\\
\hline
s_0 &  m_1 & Q_1\\
\hline
s_2s_0 &  2m_{1,1} + m_2 & Q_2\\
\hline
s_1s_2s_0 & 2m_{1,1,1} + m_{2,1} & Q_{21}\\
\hline
s_3s_2s_0 & 2m_{1, 1, 1} + m_{2, 1} + \frac{1}{2}m_3 & 2Q_3\\
\hline
s_1s_3s_2s_0 &4m_{1, 1, 1, 1} + 2m_{2, 1, 1} + m_{2, 2} + \frac{1}{2}m_{3, 1}  & 2Q_{31}\\
\hline
s_2s_3s_2s_0 & 4m_{1, 1, 1, 1} + 2m_{2, 1, 1} + m_{2, 2} + m_{3, 1} + \frac{1}{2}m_4 & 2Q_4\\
\hline
s_2s_1s_3s_2s_0 & 8m_{1, 1, 1, 1, 1} + 4m_{2, 1, 1, 1} + 2m_{2, 2, 1} + m_{3, 1, 1} + \frac{1}{2}m_{3, 2} & 2Q_{3,2} + 2Q_{4,1}\\
\hline
s_1s_2s_3s_2s_0 & 4m_{1, 1, 1, 1, 1} + 2m_{2, 1, 1, 1} + m_{2, 2, 1} + m_{3, 1, 1} + \frac{1}{2}m_{3, 2} + \frac{1}{2}m_{4, 1} & 2Q_{4,1} + 2Q_5 \\
\hline
s_0s_2s_3s_2s_0 & 4m_{1, 1, 1, 1, 1} + 2m_{2, 1, 1, 1} + m_{2, 2, 1} + m_{3, 1, 1} + \frac{1}{2}m_{3, 2} + \frac{1}{2}m_{4, 1} + \frac{1}{2}m_{5} & 2Q_5\\
\hline
\end{array}
$$
\normalsize

\newpage

\bibliographystyle{acm}
\bibliography{assf_bibliography}
\end{document}